\documentclass[reqno,a4paper,twoside]{amsart}

\usepackage[a4paper,margin=0.8in,footskip=0.25in]{geometry}

\usepackage{graphicx}
\usepackage{enumitem}

\setenumerate{label=\textnormal{(\arabic*)}}

\allowdisplaybreaks[4]

\usepackage{amsmath,amssymb,dsfont,verbatim,bm,mathrsfs,stmaryrd,mathabx}

\usepackage{mathtools}
\usepackage[raggedright]{titlesec}
\usepackage[utf8]{inputenc}
\usepackage[T1]{fontenc}

\usepackage[x11names,dvipsnames]{xcolor}
	\definecolor{egyptianblue}{rgb}{0.06, 0.2, 0.65}
	\definecolor{green(ncs)}{rgb}{0.0, 0.62, 0.42}

\usepackage{hyperref}
\hypersetup{
	colorlinks=true,
	citecolor=green(ncs),
	linkcolor= egyptianblue,
	filecolor=magenta,
	urlcolor=cyan,
	unicode=true,
	pdfpagemode=FullScreen,
}

\titleformat{\chapter}[display]
{\normalfont\huge\bfseries}{\chaptertitlename\\thechapter}{20pt}{\Huge}
\titleformat{\section}
{\normalfont\Large\bfseries\center}{\thesection}{1em}{}
\titleformat{\subsection}
{\normalfont\large\bfseries}{\thesubsection}{1em}{}
\titleformat{\subsubsection}
{\normalfont\normalsize\bfseries}{\thesubsubsection}{1em}{}
\titleformat{\paragraph}[runin]
{\normalfont\normalsize\bfseries}{\theparagraph}{1em}{}
\titleformat{\subparagraph}[runin]
{\normalfont\normalsize\bfseries}{\thesubparagraph}{1em}{}

\titlespacing*{\chapter} {0pt}{50pt}{40pt}
\titlespacing*{\section} {0pt}{3.5ex plus 1ex minus .2ex}{2.3ex plus .2ex}
\titlespacing*{\subsection} {0pt}{3.25ex plus 1ex minus .2ex}{1.5ex plus .2ex}
\titlespacing*{\subsubsection}{0pt}{3.25ex plus 1ex minus .2ex}{1.5ex plus .2ex}
\titlespacing*{\paragraph} {0pt}{3.25ex plus 1ex minus .2ex}{1em}
\titlespacing*{\subparagraph} {\parindent}{3.25ex plus 1ex minus .2ex}{1em}

\usepackage{float, subfig}
\usepackage{amsrefs}

\usepackage{titletoc}

\contentsmargin{2.55em}
\dottedcontents{section}[1.5em]{}{2em}{1pc}
\dottedcontents{subsection}[4.35em]{}{2.8em}{1pc}
\dottedcontents{subsubsection}[7.6em]{}{3.2em}{1pc}
\dottedcontents{paragraph}[10.3em]{}{3.2em}{1pc}

\usepackage[columns=3,rule=0pt]{idxlayout}
\setindexprenote{For convenience of the reader we list below almost all symbols used in this work, indicating the page in which each of them is introduced. We only except those that are extremely standard.}

\makeindex

\newtheorem{theorem}{Theorem}[section]
\newtheorem{lemma}[theorem]{Lemma}
\newtheorem{proposition}[theorem]{Proposition}
\newtheorem{corollary}[theorem]{Corollary}

\theoremstyle{definition}
\newtheorem{definition}[theorem]{Definition}

\newtheorem{notation}[theorem]{Notation}

\theoremstyle{remark}
\newtheorem{remark}[theorem]{Remark}

\usepackage{tikz}
\usetikzlibrary{matrix,shapes}
\usetikzlibrary{cd}
\tikzcdset{diagrams={nodes={inner sep=2pt}}}
\usetikzlibrary{arrows}
\tikzcdset{arrow style=tikz, diagrams={>={Latex[round,length=3pt,width=2pt,inset=1.75pt]}}}

\DeclareMathOperator{\Ext}{Ext}
\DeclareMathOperator{\EXT}{EXT}

\DeclareMathOperator{\Soc}{Soc}

\DeclareMathOperator{\Z}{Z}

\DeclareMathOperator{\Hom}{Hom}

\DeclareMathOperator{\Tot}{Tot}

\DeclareMathOperator{\sh}{sh}

\DeclareMathOperator{\sg}{sg}

\DeclareMathOperator{\ho}{H}

\newcommand{\ov}{\overline}
\newcommand{\ot}{\otimes}
\newcommand{\wh}{\widehat}

\newcommand{\xcirc}{\circ}

\numberwithin{equation}{section}

\DeclareMathAlphabet{\mathpzc}{OT1}{pzc}{m}{it}

\usepackage{mathtools}
\newtagform{brackets}{[}{]}
\usetagform{brackets}

\let\origmaketitle\maketitle
\def\maketitle{
	\begingroup
	\def\uppercasenonmath##1{} % this disables uppercasing title
	\let\MakeUppercase\relax % this disables uppercasing authors
	\origmaketitle
	\endgroup
}

\usepackage{fancyhdr}
\usepackage[foot]{amsaddr}

\usepackage{adjustbox}

\begin{document}
	
\title{\large Extensions of Linear Cycle Sets}

\author{Jorge A. Guccione}
\address{Departamento de Matem\'atica\\ Facultad de Ciencias Exactas y Naturales-UBA, Pabell\'on~1-Ciudad Universitaria\\ Intendente Guiraldes 2160 (C1428EGA) Buenos Aires, Argentina.}
\address{Instituto de Investigaciones Matem\'aticas ``Luis A. Santal\'o''\\ Pabell\'on~1-Ciudad Universitaria\\ Intendente Guiraldes 2160 (C1428EGA) Buenos Aires, Argentina.}
\email{vander@dm.uba.ar}
	
\author{Juan J. Guccione}
\address{Departamento de Matem\'atica\\ Facultad de Ciencias Exactas y Naturales-UBA, Pabell\'on~1-Ciudad Universitaria\\ Intendente Guiraldes 2160 (C1428EGA) Buenos Aires, Argentina.}
\address{Instituto Argentino de Matem\'atica-CONICET, Saavedra 15 3er piso\\ (\!C1083ACA\!) Buenos Aires, Argentina.}
\email{jjgucci@dm.uba.ar}
	
\author[C. Valqui]{Christian Valqui}
\address{Pontificia Universidad Cat\'olica del Per\'u - Instituto de Matem\'atica y Ciencias Afi\-nes, Secci\'on Matem\'aticas, PUCP, Av. Universitaria 1801, San Miguel, Lima 32, Per\'u.}
\email{cvalqui@pucp.edu.pe}

	
\begin{abstract}
We generalize the cohomology theory for linear cycle sets introduced by Lebed and Vendramin. Our cohomology  classifies extensions of linear cycle sets by trivial ideals, whereas the cohomology of Lebed and Vendramin only deals with central ideals $I$ (which are automatically trivial). Therefore our theory gives an analog to the theory of extensions of braces by trivial ideals constructed by Bachiller, but from a cohomological point of view. We also study the general notions of extensions of linear cycle sets and the equivalence of extensions.
\end{abstract}

%\begin{nouppercase}	
\maketitle
%\end{nouppercase}
	
\tableofcontents
	
\section*{Introduction}	

The celebrated Yang–Baxter equation was first established by Yang ~\cite{Y} in 1967 and then by Baxter~\cite{B} in 1972, motivated by problems in quantum physics. Since then, many solutions of various forms of the
Yang–Baxter equation have been constructed by physicists and mathematicians. In the past three decades this equation has been widely studied from different
perspectives and attracted the attention of a wide range of mathematicians because of the applications to knot theory, representations of braid groups, Hopf algebras, quantum groups, etcetera.

The relevance of this equation in mathematics and physics and the fact that the construction of new solutions is widely open, led Drinfeld~\cite{Dr} to ask for the family of set-theoretical solutions, which are relevant because of its relations with other important mathematical structures such that Affine torsors, Solvable grous, Biberbach groups and groups of I-type, Artin-Schelter regular rings, Garside structures, biracks, Hopf algebras, left symmetric algebras, etcetera; see for example~\cites{CR, CJO2, CJR, DG, D, GI, GI2, GIVB, JO, S}. An important class of set-theoretical solutions of the Yang-Baxter equation are the non-degenerated involutive solutions. The study of these solutions led to the introduction of the following mathematical structures: bijective $1$-cocyles, braces and linear cycle sets~\cites{CJO, ESS, LV, R1, R3} (which, in fact, are different avatars of the same notion).

This paper is written in the language of linear cycle sets, which are abelian groups $(A,+)$, endowed with another binary operation satisfying suitable conditions. This structure was introduced by  Rump in~\cite{R1}.

In~\cite{B} Bachiller started the study of extensions of linear cycle sets by trivial linear cycle sets (in the language of braces) due to its relation with the classification problem of these structures. An alternative approach to extensions was suggested earlier in~\cite{DG}), and it was translated to the language of braces by Bachiller in~\cite{B1}. In parallel, and more or less at the same time, in~\cite{LV} the authors define the notion of central extensions of linear cycle sets and introduce a cohomology theory that classifies them.

The aim of this paper is to introduce the general notion of extension of a linear cycle set~$H$ by a linear cycle set~$I$ and to prove that the extensions of~$H$ by a trivial linear cycle set~$I$ are classified by the second cohomology group of a cochain complex, generalizing the main results in~\cite{LV}. Our main results are Corollary~\ref{ppal seccion 4}, Theorem~\ref{principal}, Corollary~\ref{equivalencia}, Theorem~\ref{complejo caso general} and Corollary~\ref{equiv caso general}.

This paper is organized as follows: In Section~\ref{section: Preliminaries} we review quickly the definitions of brace and linear cycle set and the relation between these notions. We also review the definition of ideal of linear cycle set. In Section~\ref{Extensions of linear cycle sets}, we give the general notion of an extension
\begin{equation}\label{extensiones en la introduccion}
\begin{tikzpicture}
\begin{scope}[yshift=0cm,xshift=0cm, baseline]
		\matrix(BPcomplex) [matrix of math nodes, row sep=0em, text height=1.5ex, text
		depth=0.25ex, column sep=2.5em, inner sep=0pt, minimum height=5mm, minimum width =6mm]
		{0 & I & B & H & 0,\\};
		\draw[-{latex}] (BPcomplex-1-1) -- node[above=1pt,font=\scriptsize] {} (BPcomplex-1-2);
		\draw[-{latex}] (BPcomplex-1-2) -- node[above=1pt,font=\scriptsize] {$\iota$} (BPcomplex-1-3);
		\draw[-{latex}] (BPcomplex-1-3) -- node[above=1pt,font=\scriptsize] {$\pi$} (BPcomplex-1-4);
		\draw[-{latex}] (BPcomplex-1-4) -- node[above=1pt,font=\scriptsize] {} (BPcomplex-1-5);
\end{scope}
\end{tikzpicture}
\end{equation}
of a linear cycle set $H$ by a linear cycle set $I$, and we define when two such extensions are equivalent, mimicking the classical way of presenting these notions for abelian groups. Then, given an extension~\eqref{extensiones en la introduccion}, and a set-theoretic section $s$ of $\pi$ with $s(0)=0$, we obtain weak actions $\blackdiamond\colon H\times I\to I$ and $\Yleft \colon I\times H\to I$ and (in general non cohomological) cocycles $\beta\colon H\times H\to I$ and $f\colon H\times H\to I$ that allow us to express the structure operations of $B$ in terms of these maps and the structure operations of $H$ and $I$ (see formulas \eqref{construccion de beta} and \eqref{formula para cdot en I times H}). Moreover, we prove that when $I$ is trivial, the maps $\blackdiamond$ and $\Yleft$ do not depend on $s$. In Section~\ref{Building extensions of linear cycle sets}, given maps $\blackdiamond$, $\Yleft$, $\beta$ and $f$ as above, we obtain necessary and sufficient conditions that they must satisfy in order that $I\times H$ becomes a linear cycle set via \eqref{construccion de beta} and \eqref{formula para cdot en I times H} (with $s(h)=w_h\coloneqq (0,h)$). Moreover, we obtain conditions that guarantee that two such extensions are equivalent. In Section~\ref{Extensions with central cocycles} we assume that $h\blackdiamond (y\Yleft h')$, $h\blackdiamond \beta(h',h'')$ and $h\blackdiamond f(h',h'')$ belong to the center of $I$ for all $h,h',h''\in H$ and $y\in I$, and we prove that under this hypothesis the conditions in Section~\ref{Building extensions of linear cycle sets} can be considerably simplified. Moreover, when $I$ is trivial we compare our maps $\blackdiamond$ and $\Yleft$ with the maps $\sigma$ and $\nu$ introduced in~\cite{B}. In Section~\ref{Cohomology of linear cycle sets1}, for a linear cycle set $H$, a trivial linear cycle set $I$ and a map $\blackdiamond$ satisfying suitable conditions, we construct a double cochain complex whose second cohomology group classifies the extensions of $H$ by $I$ with $\Yleft=0$ and action $\blackdiamond$ (these are the extensions~\eqref{extensiones en la introduccion} with $I$ included in the socle of~$B$). For central extensions (which are those in which $\blackdiamond$ is the trivial action and $\Yleft=0$) this complex coincides with the normalized complex obtained in~\cite{LV}*{Section 4}. Finally, in Section~\ref{Cohomology of linear cycle sets} we remove the hypothesis that $\Yleft=0$ (when $\Yleft\ne 0$ the obtained cochain complex is no longer a double complex).

\smallskip

The following diagram illustrates the relations between the different sections of this paper and the connections with the articles~\cite{B} and~\cite{LV}.
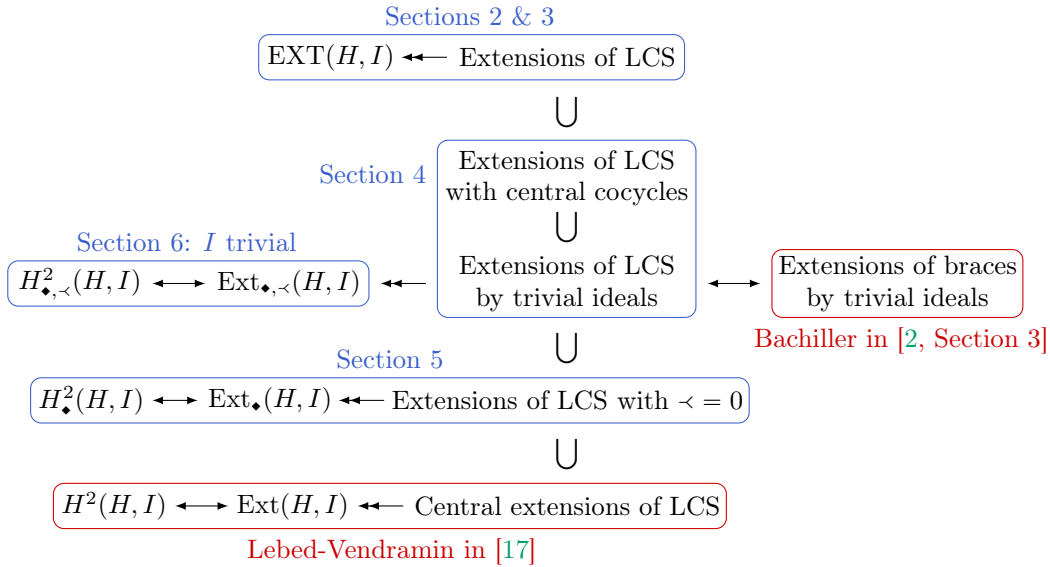
\begin{center}
\begin{tikzpicture}
\draw(4.81,7.6) node {\textrm{Extensions of LCS}};
\draw(4.81,6.88) node[rotate=90] {\scalebox{1.8}{$\subset$}};
\draw(4.8,6) node {$\begin{array}{c} \textrm{Extensions of LCS} \\ \textrm{with central cocycles} \end{array}$};
\draw(4.81,5.35) node[rotate=90] {\scalebox{1.8}{$\subset$}};
\draw(4.8,4.6) node {$\begin{array}{c} \textrm{Extensions of LCS} \\ \textrm{by trivial ideals} \end{array}$};
\draw(4.81,3.75) node[rotate=90] {\scalebox{1.8}{$\subset$}};
\draw(4.81,3.05) node {\textrm{Extensions of LCS with $\Yleft=0$}};
\draw(4.81,2.35) node[rotate=90] {\scalebox{1.8}{$\subset$}};
\draw(4.81,1.65) node {\textrm{Central extensions of LCS}};
\draw(1.7,7.6) node {$\EXT(H,I)$};
\draw[-{latex}{latex}] (3.225,7.6)--(2.625,7.6);
\draw(1.15,4.6) node {$\Ext_{\blackdiamond,\Yleft}(H,I)$};
\draw[-{latex}{latex}] (2.95,4.6)--(2.35,4.6);
\draw(0.9,3.05) node {$\Ext_{\blackdiamond}(H,I)$};
\draw[-{latex}{latex}] (2.4,3.05)--(1.8,3.05);
\draw(1.2,1.65) node {$\Ext(H,I)$};
\draw[-{latex}{latex}] (2.65,1.65)--(2.05,1.65);
\draw(-1.6,4.6) node {$H_{\blackdiamond,\Yleft}^2(H,I)$};
\draw[<->, latex-latex] (-0.65,4.6)--(0.05,4.6);
\draw(-1.45,3.05) node {$H_{\blackdiamond}^2(H,I)$};
\draw[<->, latex-latex] (-0.65,3.05)--(-0.05,3.05);
\draw(-1.15,1.65) node {$H^2(H,I)$};
\draw[<->, latex-latex] (-0.35,1.65)--(0.35,1.65);
\draw[<->, latex-latex] (6.65,4.6)--(7.35,4.6);
\draw (9.15,4.6) node {$\begin{array}{c} \textrm{Extensions of braces} \\ \textrm{by trivial ideals} \end{array}$};
\draw[rounded corners,Red3] (7.5,4.15) rectangle (10.8,5.05) {};
\draw[Red3] (9.2,3.85) node{Bachiller in~\cite{B}*{Section 3}};
\draw[rounded corners,Red3] (-1.95,1.35) rectangle (6.9,1.95) {};
\draw[Red3] (2.5,1.05) node{Lebed-Vendramin in~\cite{LV}};
\draw[rounded corners,RoyalBlue3] (-2.25,2.75) rectangle (7.2,3.35) {};
\draw[RoyalBlue3] (2.475,3.6) node{Section 5};
\draw[rounded corners,RoyalBlue3] (-2.55,4.3) rectangle (2.2,4.9) {};
\draw[RoyalBlue3] (-0.2,5.15) node{Section 6: $I$ trivial};
\draw[rounded corners,RoyalBlue3] (3.1,4.15) rectangle (6.5,6.5) {};
\draw[RoyalBlue3] (2.25,6.05) node{Section 4};
\draw[rounded corners,RoyalBlue3] (0.75,7.3) rectangle (6.35,7.9) {};
\draw[RoyalBlue3] (3.525,8.15) node{Sections 2 \& 3};
\end{tikzpicture}
\end{center}

\paragraph{Precedence of operations} The operations precedence in this paper is the following: the operators with the highest precedence are the unary operation $a\mapsto a^{-1}$ and the binary operation $(a,b)\mapsto a^b$; then comes the multiplication $(a, b)\mapsto ab$; then, the operations $\cdot$, $\blackdiamond$ and $\triangleleft$, that have equal precedence; then the operator $\Yleft$, and finally, the sum. Of course, as usual, this order of precedence can be modified by the use of parenthesis.
	
\section{Preliminaries}\label{section: Preliminaries}
A {\em left brace} is an abelian group $(A,+)$ endowed with an additional group operation $(a,b)\mapsto ab$, such that
\begin{equation}\label{compatibilidad braza}
a(b+c)= ab+ac-a\quad\text{for all $a,b,c\in A$.}
\end{equation}
The two group structures necessarily share the same neutral element, denoted by $0$. Of course, the additive inverse of $a$ is denoted by $-a$, while its multiplicative inverse is denoted by $a^{-1}$. In particular $0^{-1} = 0$.
	
\smallskip
	
A {\em linear cycle set} is an abelian group $(A,+)$ endowed with a binary operation $\cdot$, having bijective left translations $a\mapsto b\cdot a$, and satisfying the conditions
\begin{align}
&a\cdot (b+c)=a\cdot b + a\cdot c\label{compatibilidad cdot suma}
\shortintertext{and}
&(a+b)\cdot c = (a\cdot b)\cdot (a\cdot c),\label{compatibilidad suma cdot}
\end{align}
for all $a,b,c\in A$. It is well known that each linear cycle set is a cycle set. That is:
\begin{equation}
(a\cdot b)\cdot (a\cdot c)=(b\cdot a)\cdot (b\cdot c)\quad\text{for all $a,b,c\in A$}\label{condicion de cycle set}.
\end{equation}
The notions of left braces and linear cycle sets were introduced by Rump, who proved that they are equivalent via the relation
\begin{equation}
a\cdot b = a^{-1}(a+b)\quad \text{and}\quad ab = {}^ab + a,\label{equivalencia linear cycle set braza}
\end{equation}
where the map $b\mapsto {}^ab$ is the inverse of the map $b\mapsto a\cdot b$. In Section~\ref{Cohomology of linear cycle sets}, we will use that, for all $a,b,c\in A$,
\begin{equation}\label{formula que usaremos}
a\cdot(b+c) = a\cdot b+a\cdot c = (a\cdot b)((a\cdot b)\cdot (a\cdot c))  = (a\cdot b)((a+b)\cdot c).
\end{equation}
	
Each additive abelian group $A$ is a linear cycle set via $a\cdot b\coloneqq b$. These ones are the so called {\em trivial linear cycle sets}. Note that by~\eqref{equivalencia linear cycle set braza}, a linear cycle set $A$ is trivial if and only if $ab = a+b$ for all $a,b\in A$.

\smallskip

Let $A$ be a linear cycle set. An {\em ideal} of $A$ is  a subgroup $I$ of $(A,+)$ such that $a\cdot y \in I$ and $y\cdot a - a\in I$ for all $a\in A$ and $y\in I$. An ideal $I$ of $A$ is called a {\em central ideal} if $y\cdot a=a$ and $a\cdot y=y$ for all $y\in I$ and $a\in A$. The {\em socle} of $A$ is the ideal $\Soc(A)$ of all $y\in A$ such that $y\cdot a = a$ for all $a\in A$. The {\em center} of $A$, is the set $Z(A)$ of all $y\in \Soc(A)$ such that $a\cdot y = y$ for all $a\in A$. The center of $A$ is a central ideal of $A$, and each central ideal of $A$ is included in $Z(A)$.

\begin{notation}\label{Yleft} Given an ideal $I$ of a linear cycle set $(A,+,\cdot)$, for each $y\in I$ and $a\in A$ we set $y \Yleft a\coloneqq y\cdot a - a$.
\end{notation}

\section{Extensions of linear cycle sets}\label{Extensions of linear cycle sets}
In this section we define the notion of extensions of linear cycle sets mimicking the notion of extensions of abelian groups. Our main result is the formula for $\cdot$ in~\eqref{formula para cdot en I times H}. Let $I$ and $H$ be additive abelian groups. Recall that an extension $(\iota, B, \pi)$, of $H$ by $I$, is a short exact sequence
\begin{equation}
\begin{tikzpicture}
	\begin{scope}[yshift=0cm,xshift=0cm, baseline]
		\matrix(BPcomplex) [matrix of math nodes, row sep=0em, text height=1.5ex, text
		depth=0.25ex, column sep=2.5em, inner sep=0pt, minimum height=5mm, minimum width =6mm]
		{0 & I & B & H & 0,\\};
		\draw[-{latex}] (BPcomplex-1-1) -- node[above=1pt,font=\scriptsize] {} (BPcomplex-1-2);
		\draw[-{latex}] (BPcomplex-1-2) -- node[above=1pt,font=\scriptsize] {$\iota$} (BPcomplex-1-3);
		\draw[-{latex}] (BPcomplex-1-3) -- node[above=1pt,font=\scriptsize] {$\pi$} (BPcomplex-1-4);
		\draw[-{latex}] (BPcomplex-1-4) -- node[above=1pt,font=\scriptsize] {} (BPcomplex-1-5);
	\end{scope}\label{sucesion exacta corta de grupos abelianos}
\end{tikzpicture}
\end{equation}
of additive abelian groups. Recall also that two such extensions, $(\iota, B, \pi)$ and $(\iota', B', \pi')$, are equivalent if there exists a morphism $\phi\colon B\to B'$ such that $\pi'\xcirc \phi=\pi$ and $\phi\xcirc \iota=\iota'$. Then, $\phi$ is necessarily an isomorphism.

\smallskip

Let $(\iota, B, \pi)$ be an extensions of $H$ by $I$ and let $s\colon H\to B$ be a set theoretic section of $\pi$ with $s(0)=0$. Then, for each $ b\in B$ there exist unique $y\in I$ and $h\in H$ such that $x=\iota(y)+s(h)$. Moreover, since~$\iota$ and $\pi$ are morphisms and $\iota$ is injective, there exists a unique map $\beta\colon H\times H\to I$ such that
\begin{equation}\label{construccion de beta}
\iota(y)+s(h) + \iota(y')+s(h')= \iota(y) + \iota(y')+ s(h) + s(h')= \iota(y+y' + \beta(h,h')) + s(h+h').
\end{equation}
A direct computation shows that
\begin{align}
& \beta(h,h')+\beta(h+h',h'') = \beta(h',h'')+\beta(h,h'+h''),\label{condicion de cociclo}\\
&\beta(h,0) = \beta(0,h) = 0\qquad\text{and}\qquad \beta(h,h') = \beta(h',h)\label{normalidad y cociclo abeliano},
\end{align}	
for all $h,h'\in H$. In fact, the first condition is equivalent to the associativity of the sum in $B$; the second one, to the fact that $s(0)=0$, and the third one, to the commutativity of the sum in $B$. Conversely, given a map $\beta\colon H\times H\to I$ satisfying~\eqref{condicion de cociclo} and~\eqref{normalidad y cociclo abeliano}, we have an extension
\begin{equation}\label{sucesion exacta corta de grupos abelianos 2}
\begin{tikzpicture}
\begin{scope}[yshift=0cm,xshift=0cm, baseline]
		\matrix(BPcomplex) [matrix of math nodes, row sep=0em, text height=1.5ex, text
		depth=0.25ex, column sep=2.5em, inner sep=0pt, minimum height=5mm, minimum width =6mm]
		{0 & I & I\times_{\beta} H & H & 0,\\};
		\draw[->] (BPcomplex-1-1) -- node[above=1pt,font=\scriptsize] {} (BPcomplex-1-2);
		\draw[->] (BPcomplex-1-2) -- node[above=1pt,font=\scriptsize] {$\iota$} (BPcomplex-1-3);
		\draw[->] (BPcomplex-1-3) -- node[above=1pt,font=\scriptsize] {$\pi$} (BPcomplex-1-4);
		\draw[->] (BPcomplex-1-4) -- node[above=1pt,font=\scriptsize] {} (BPcomplex-1-5);
\end{scope}
\end{tikzpicture}
\end{equation}
of $H$ by $I$, in which  $I\times_{\beta} H$ is $I\times H$ endowed with the sum
$$
(y,h)+(y',h')\coloneqq (y+y'+\beta(h,h'),h+h'),
$$
and the maps $\iota$ and $\pi$ are the canonical ones. For the sake of simplicity, here and subsequently we will write $y$ instead of $(y,0)$ and we set $w_h\coloneqq (0,h)$. Thus,
$(y,h)=(y,0)+(0,h)=y+w_h$. With this notation $\iota(y) = y$, $\pi(y+w_h) =h$ and the sum in $I\times_{\beta} H$ reads
\begin{equation}\label{formula para suma en I times H}
(y+w_h) + (y'+w_{h'}) = y+y'+\beta(h,h')+ w_{h+h'}.
\end{equation}
It is evident that each extension of abelian groups$(\iota,B,\pi)$, endowed with a set theoretic section $s$ of $\pi$ such that $s(0)=0$, is equivalent to the extension \eqref{sucesion exacta corta de grupos abelianos 2} associated with the map $\beta$ obtained in~\eqref{construccion de beta}. Concretely, the map $\phi\colon I\times_{\beta} H\to B$ is given by $\phi(y+w_h)=\iota(y)+s(h)$. Moreover, it is well known that two extensions $(\iota,I\times_{\beta} H,\pi)$ and $(\iota,I\times_{\beta'} H,\pi)$ are equivalent if and only if there exists a map $\varphi\colon H\to I$, such that
\begin{equation}\label{equivalencia a nivel de qrupos}
\varphi(0) = 0\qquad\text{and}\qquad \varphi(h)-\varphi(h+h')+\varphi(h')=\beta(h,h') - \beta'(h,h')\quad\text{for all $h,h'\in H$.}
\end{equation}
Concretely, the map $\phi\colon I\times_{\beta} H\to I\times_{\beta'} H$ performing the equivalence is given by $\phi(y+w_h)= y+\varphi(h)+w_h$.

\begin{definition}\label{extensiones de linear cycle sets} Let $I$ and $H$ be linear cycle sets. An {\em extension} $(\iota,B, \pi)$, of $H$ by $I$, is a short exact sequence like~\eqref{sucesion exacta corta de grupos abelianos}, in which $B$ is a linear cycle set and both $\iota$ are $\pi$ are linear cycle sets morphisms. Two extensions $(\iota,B,\pi)$ and $(\iota',B',\pi')$, of $H$ by $I$, are {\em equivalent} if there exists a linear cycle set morphism $\phi\colon B\to B'$ such that $\pi'\xcirc \phi=\pi$ and $\phi\xcirc \iota=\iota'$. As for extension of groups, then $\phi$ is an isomorphism. We let $\EXT(H;I)$ denote the set of equivalence classes of extensions $(\iota,B,\pi)$, of $H$ by $I$.
\end{definition}

\subsection[Formula for \texorpdfstring{$\cdot$}{.}]{Formula for \texorpdfstring{$\pmb{\cdot}$}{.}}\label{formula for cdot}
Assume that~\eqref{sucesion exacta corta de grupos abelianos} is an extension of linear cycle sets and let $s$ be a section of $\pi$ with $s(0)=0$. For the sake of simplicity in the rest of this section we identify $I$ with $\iota(I)$ and we write $y$ instead of $\iota(y)$. Since, by~\eqref{compatibilidad cdot suma},
$$
(y+s(h))\cdot (y'+s(h')) = (y+s(h))\cdot y' + (y+s(h))\cdot s(h'),
$$
we can study the action $\cdot$ of $B$ on elements of $I$ and $s(H)$ separately.

\paragraph{Action on elements of $\mathbf{I}$} Since by~\eqref{compatibilidad suma cdot} and~\eqref{condicion de cycle set}
$$
(y+s(h))\cdot y' = (y\cdot s(h))\cdot (y\cdot y') = (s(h)\cdot y)\cdot (s(h)\cdot y'),
$$
the action $\cdot$ of $B$ on $I$ is determined by the operation $\cdot$ of $I$ and the map $\blackdiamond\colon H\times I\to I$ given by $h\blackdiamond y \coloneqq s(h)\cdot y$.

\paragraph{Action on elements of $\mathbf{s(H)}$} Since $\pi\colon B\to H$ is a linear cycle set morphism, there is a map $f\colon H\times H\to I$, such that
$$
s(h)\cdot s(h') = f(h,h')+ s(h\cdot h')\quad\text{for all $h,h'\in H$.}
$$
By \eqref{compatibilidad cdot suma}, \eqref{compatibilidad suma cdot} and~\eqref{condicion de cycle set}, we have
\begin{align*}
(y+s(h))\cdot s(h')  & = (s(h)\cdot y)\cdot (s(h)\cdot s(h'))\\
& = (h\blackdiamond y)\cdot (f(h,h') + s(h\cdot h'))\\
& = (h\blackdiamond y)\cdot f(h,h') + (h\blackdiamond y)\cdot s(h\cdot h')\\
& = (h\blackdiamond y)\cdot f(h,h') +  h\blackdiamond y\Yleft s(h\cdot h') + s(h\cdot h'),
\end{align*}

\paragraph{Formula for $\pmb{\cdot}$ on $\mathbf{B}$}
From now on, for $y\in I$ and $h\in H$, we set $y\Yleft h\coloneqq y\Yleft s(h) = y\cdot s(h)-s(h)$. By the discussion in the previous paragraphs, for all $y,y'\in I$ and $h,h'\in H$, we have:
\begin{equation}\label{formula para cdot en I times H}
(y+s(h))\cdot (y'+s(h'))= (h\blackdiamond y)\cdot (h\blackdiamond y') + (h\blackdiamond y) \cdot f(h,h') + h\blackdiamond y \Yleft h\cdot h' + s(h\cdot h').
\end{equation}

\begin{proposition}\label{primeras propiedades} The following properties hold:

\begin{enumerate}[itemsep=0.0ex, topsep=1.0ex, label=\emph{(\arabic*)}]

\item $h\blackdiamond (y+y') = h\blackdiamond  y+h\blackdiamond  y'$, for all $h\in H$ and $y,y'\in I$,

\item $0\blackdiamond  y = y$, for all $y\in I$,

\item $0\Yleft h = 0$, for all $h\in H$,

\item $y \Yleft 0 = 0$, for all $y\in I$,

\item $f(h,0) = f(0,h) = 0$, for all $h\in H$.

%\item $\tau(h,h')\cdot (hh'\blackdiamond y) = h'\blackdiamond (h\blackdiamond  y)$, for all $h,h'\in H$ and $y\in I$,

\end{enumerate}
%where $\tau(h,h')\coloneqq s(hh')^{-1}s(h)s(h')$.
\end{proposition}

\begin{proof} Items~(1)--(4) are clear. Since $I\times H$ and $H$ are linear cycle sets, we have
$$
s(h) = s(0)\cdot s(h) = f(0,h)+ s(0\cdot h) = f(0,h)+ s(h)\quad\text{and}\quad 0 = s(h)\cdot s(0) = f(h,0)+ s(h\cdot 0) = f(h,0),
$$
which implies item~(5). %Finally, we have
%$$
%\tau(h,h')\cdot (hh'\blackdiamond y) = \tau(h,h')\cdot (s(hh')\cdot y) = s(h)s(h')\cdot y = s(h')\cdot (s(h)\cdot y) = h'\blackdiamond (h\blackdiamond y),
%$$
%which proves item~(6).
\end{proof}

The maps $\blackdiamond$ and $\Yleft$ depend on the chosen section $s$ of $\pi$. In the following proposition we prove that this is not the case when $I$ is trivial.

\begin{proposition}\label{diamante e Yleft no dependen de s} Let $H$ and $I$ be linear cycle sets, $(\iota,B, \pi)$ an extension of $H$ by $I$, and $s$ and $s'$ sections~of~$\pi$, such that $s(0)=s'(0) = 0$. Let $\blackdiamond\colon H\times I\to I$, $\blackdiamond'\colon H\times I\to I$, $\Yleft \colon I\times H\to I$ and $\Yleft'\colon I\times H\to I$ be the maps defined by $h\blackdiamond y \coloneqq s(h)\cdot y$, $h\blackdiamond' y \coloneqq s'(h)\cdot y$, $y\Yleft h\coloneqq y\cdot s(h)-s(h)$ and $y\Yleft' h\coloneqq y\cdot s'(h)-s'(h)$. If $I$ is trivial, then $\blackdiamond'  = \blackdiamond$ and $\Yleft' =\Yleft$.
\end{proposition}

\begin{proof} For $h\in H$ write $j(h)\coloneqq s'(h)-s(h)$. Note that $j(h)\in\iota(I)$. By identities~\eqref{compatibilidad cdot suma} and~\eqref{compatibilidad suma cdot}, the fact that $I$ is an ideal of $B$, and $I$ is trivial, we have
\begin{align*}
&h \blackdiamond' y = s'(h)\cdot y = (s(h)+j(h))\cdot y = (s(h)\cdot j(h))\cdot (s(h)\cdot y) = s(h)\cdot y = h \blackdiamond y
\shortintertext{and}
&y \Yleft' h = y \cdot s(h)- s(h) + y \cdot j(h) - j(h) = y \cdot s(h)- s(h) y \Yleft' h,
\end{align*}
as desired.
\end{proof}

\begin{notation}\label{notacion ext} Let $H$ and $I$ be linear cycle sets and assume that $I$ is trivial. Fix maps $\blackdiamond\colon H\times I\to I$ and $\Yleft \colon I\times H\to I$. We let $\Ext_{\blackdiamond,\Yleft}(H;I)$ denote the set of equivalence classes of  extensions $(\iota,B,\pi)$, of $H$ by $I$, such that $s(h)\cdot y = h\blackdiamond y$ and $y\cdot s(h) = y\Yleft h$, for every section $s$ of $\pi$ satisfying $s(0) = 0$. When $\Yleft$ is the zero map, we write $\Ext_{\blackdiamond}(H;I)$ instead of $\Ext_{\blackdiamond,0}(H;I)$.
\end{notation}

\section{Building extensions of linear cycle sets}\label{Building extensions of linear cycle sets}

In this section
\begin{equation}\label{extension}
\begin{tikzpicture}
\begin{scope}[yshift=0cm,xshift=0cm, baseline]
		\matrix(BPcomplex) [matrix of math nodes, row sep=0em, text height=1.5ex, text
		depth=0.25ex, column sep=2.5em, inner sep=0pt, minimum height=5mm, minimum width =6mm]
		{0 & I & I\times_{\beta} H & H & 0\\};
		\draw[->] (BPcomplex-1-1) -- node[above=1pt,font=\scriptsize] {} (BPcomplex-1-2);
		\draw[->] (BPcomplex-1-2) -- node[above=1pt,font=\scriptsize] {$\iota$} (BPcomplex-1-3);
		\draw[->] (BPcomplex-1-3) -- node[above=1pt,font=\scriptsize] {$\pi$} (BPcomplex-1-4);
		\draw[->] (BPcomplex-1-4) -- node[above=1pt,font=\scriptsize] {} (BPcomplex-1-5);
\end{scope}
\end{tikzpicture}
\end{equation}
is a short exact sequence of abelian groups as in~\eqref{sucesion exacta corta de grupos abelianos 2} (that is,~\eqref{condicion de cociclo} and~\eqref{normalidad y cociclo abeliano} are fulfilled). We assume that $I$ and $H$ are linear cycle sets and we consider maps $\blackdiamond\colon H\times I\to I$, $\Yleft \colon I\times H\to I$ and $f\colon H\times H\to I$ satisfying conditions~(1)--(4) of Proposition~\ref{primeras propiedades}. Moreover, we let $\cdot$ denote the binary operation of $I\times_{\beta} H$ defined as in~\eqref{formula para cdot en I times H} (with $s(h) = w_h$, etcetera). In what follows we are going to use these properties without explicit mention. In Proposition~\ref{caracterizacion de extensiones} we obtain necessary and sufficient conditions that guarantee that~\eqref{extension} is an extension of linear cycle sets. Then, in Proposition~\ref{equivalecia entre extensiones producto}, we determine when two of such extensions are equivalent. Moreover, as usual, any extension of linear cycle set is equivalent to one of these (see Remark~\ref{toda extension es equivalente a una...}).

\begin{remark} For all $h,h'\in H$, we have $h\blackdiamond 0 = 0$ and $w_h\cdot w_{h'} = f(h,h')+ w_{h\cdot h'}$.
\end{remark}

\begin{proposition}\label{segundas propiedades} The equalities $y\cdot w_h = y\Yleft h + w_h$ and $w_h\cdot y = h\blackdiamond y$ hold for all $h\in H$ and $y\in I$ if and only if $f(h,0)=f(0,h)=0$ for all $h\in H$.
\end{proposition}

\begin{proof} This follows by equality~\eqref{formula para cdot en I times H} and Proposition~\ref{primeras propiedades}.
\end{proof}

In the rest of this section we assume that $f(h,0)=f(0,h)=0$ for all $h\in H$.  By this, and items (2) and (4) of Proposition~\ref{primeras propiedades}, the map $\iota$ is compatible with $\cdot$.

\begin{remark}\label{son biyectivas} We claim that the maps $y'+w_{h'} \mapsto (y+w_h)\cdot (y'+w_{h'})$ are permutations of $I\times H$ if and only if the maps $y'\mapsto h\blackdiamond y'$ are permutations of $I$. In fact $(y+w_h)\cdot (y'+w_{h'}) = y''+w_{h''}$ if and only if
$$
h\cdot h' = h''\quad\text{and}\quad (h\blackdiamond y)\cdot (h\blackdiamond y') = y'' - (h\blackdiamond y) \cdot f(h,h') + h\blackdiamond y \Yleft h\cdot h' - s(h\cdot h').
$$
Since $H$ and $I$ are linear cycle sets, there are unique $h'$ and $h\blackdiamond y'$ satisfying these identities. Now the claim follows immediately.
\end{remark}

Here and subsequently, we set
$$
y^{\beta}_{y,y',h,h'}\coloneqq y+y'+\beta(h,h')\qquad\text{and}\qquad  y^f_{y,y',h,h'}\coloneqq (h\blackdiamond y)\cdot (h\blackdiamond y') + (h\blackdiamond y) \cdot f(h,h') + h\blackdiamond y \Yleft h\cdot h'.
$$

\begin{proposition}\label{segunda equivalencia de (2.1)} The operation $\cdot$ of $I\times_{\beta} H$ satisfies equality~\eqref{compatibilidad cdot suma} if and only
\begin{align}
& h\blackdiamond  \beta(h',h'') + f(h,h'+h'') = f(h,h')+f(h,h'')+\beta(h\cdot h',h\cdot h'')\label{segunda equivalencia de (2.1)1}\\
\shortintertext{and}
& y\Yleft (h'+h'') = y\Yleft h' + y\Yleft h'' + \beta(h',h'') - y\cdot \beta(h',h''),\label{segunda equivalencia de (2.1)2}
\end{align}
for all $h,h',h''\in H$ and $y\in I$.
\end{proposition}

\begin{proof} By~\eqref{formula para suma en I times H} and~\eqref{formula para cdot en I times H}, we have
\begin{align*}
& (y+w_h) \cdot (y'+w_{h'}+y''+w_{h''}) %= (y+w_h)\cdot (y^{\beta}_{y',y'',h',h''}+ w_{h'+ h''})\\
%
%&\phantom{(y+w_h) \cdot (y'+w_{h'}+y''+w_{h''})}
= (h\blackdiamond y)\cdot (h\blackdiamond  y^{\beta}_{y',y'',h',h''})+ y^f_{y,0,h,h'+h''} + w_{h\cdot (h'+h'')}
\shortintertext{and}
&(y+w_h)\cdot (y'+w_{h'}) + (y+w_h) \cdot (y''+w_{h''}) %= y^f_{y,y',h,h'}  + w_{h\cdot h'} + y^f_{y,y'',h,h''}+ w_{h\cdot h''}\\
%
%& \phantom{(y+w_h)\cdot (y'+w_{h'}) + (y+w_h)\cdot (y''+w_{h''})}
= y^f_{y,y',h,h'} + y^f_{y,y'',h,h''} + \beta(h\cdot h',h\cdot h'') + w_{h\cdot h'+h\cdot h''}.
\end{align*}
Thus, since $H$ is a linear cycle set, the binary operation $\cdot$ of $I\times_{\beta} H$ satisfies equality~\eqref{compatibilidad cdot suma} if and only if
\begin{equation}\label{paso intermedio}
(h\blackdiamond y)\cdot (h\blackdiamond  y^{\beta}_{y',y'',h',h''})+ y^f_{y,0,h,h'+h''} = y^f_{y,y',h,h'} + y^f_{y,y'',h,h''} + \beta(h\cdot h',h\cdot h''),
\end{equation}
for all $y,y',y''\in I$ and $h,h',h''\in H$. Equality~\eqref{segunda equivalencia de (2.1)1} follows from this identity taking $y=y'=y''=0$,~while equality~\eqref{segunda equivalencia de (2.1)2} follows from the same identity taking $h=0$ and $y'=y''=0$. Conversely, assume that~Equal\-ities~\eqref{segunda equivalencia de (2.1)1} and~\eqref{segunda equivalencia de (2.1)2} hold. We claim that, for $h,h',h''\in H$ and $y\in Y$,
\begin{equation}\label{formula a usar}
(h \blackdiamond y)\cdot (h\blackdiamond \beta(h',h'')) + y^f_{y,0,h,h'+h''} = y^f_{y,0,h,h'} + y^f_{y,0,h,h''} +\beta(h\cdot h',h\cdot h'').
\end{equation}
In fact, by equality~\eqref{segunda equivalencia de (2.1)2}, we have
$$
h \blackdiamond y\Yleft (h \cdot (h'+ h'')) = h\blackdiamond y\Yleft h\cdot h' + h \blackdiamond y\Yleft h\cdot h'' + \beta(h\cdot h',h\cdot h'') - (h \blackdiamond y)\cdot \beta(h\cdot h',h\cdot h''),
$$
which shows that~\eqref{formula a usar} holds if and only if
$$
(h\blackdiamond y)\cdot (h\blackdiamond  \beta(h',h'')) + (h\blackdiamond y)\cdot f(h,h'+h'') - (h\blackdiamond y)\cdot \beta(h\cdot h',h\cdot h'') = (h\blackdiamond y)\cdot f(h,h') + (h\blackdiamond y)\cdot f(h,h'').
$$
Since this is a consequence of equality~\eqref{segunda equivalencia de (2.1)1}, the claim is true. The identity in~\eqref{paso intermedio} follows immediately from equality~\eqref{formula a usar} and the fact that $h\blackdiamond  (y'+y'') = h\blackdiamond y' + h\blackdiamond y''$.
\end{proof}

\begin{proposition}\label{equivalencia de (a+b).c = (a.b).(a.c)} Assume that the binary operation $\cdot$ of $I\times_{\beta} H$ satisfies con\-di\-tion~\eqref{compatibilidad cdot suma}. Then $\cdot$ satisfies~con\-di\-tion~\eqref{compatibilidad suma cdot} if and only if, for all $y,y',y''\in I$ and $h,h',h''\in H$,
\begin{align}\label{condicion equivalente compatibilidad suma cdot con y}
&((h+h') \blackdiamond y^{\beta}_{y,y',h,h'})\cdot ((h+h')\blackdiamond y'') = ((h\cdot h')\blackdiamond y^f_{y,y',h,h'})\cdot ((h\cdot h')\blackdiamond  ((h\blackdiamond y)\cdot (h\blackdiamond y'')))
\shortintertext{and}
&\begin{aligned}\label{condicion equivalente compatibilidad suma cdot con h}
& ((h+h') \blackdiamond y^{\beta}_{y,y',h,h'})  \cdot f(h+h',h'') +  (h+h')\blackdiamond y^{\beta}_{y,y',h,h'}\Yleft (h+h')\cdot h''\\
& = ((h\cdot h')\blackdiamond y^f_{y,y',h,h'})\cdot \bigl((h\cdot h')\blackdiamond y^f_{y,0,h,h''} + f(h\cdot h',h\cdot h'')\bigr) + (h\cdot h')\blackdiamond y^f_{y,y',h,h'}\Yleft (h\cdot h')\cdot(h\cdot h''). 		
\end{aligned}
\end{align}
\end{proposition}

\begin{proof} A direct computation using~\eqref{compatibilidad cdot suma} proves that condition~\eqref{compatibilidad suma cdot} is satisfied if and only if, for all $y,y',y''\in I$ and $h,h',h''\in H$,
\begin{align}
&(y+ w_h + y'+w_{h'})\cdot y'' = ((y+ w_h)\cdot (y'+w_{h'}))\cdot ((y+ w_h)\cdot y'')\label{compatibilidad suma cdot con y}
\shortintertext{and}
&(y+ w_h + y'+w_{h'})\cdot w_{h''} = ((y+ w_h)\cdot (y'+w_{h'}))\cdot ((y+ w_h)\cdot w_{h''})\label{compatibilidad suma cdot con h}.
\end{align}
So, the result follows since, by the very definitions of the operations in $I\times_{\beta} H$, the fact that $h\blackdiamond 0 = 0$ and condition~\eqref{compatibilidad cdot suma}, conditions~\eqref{compatibilidad suma cdot con y} and~\eqref{compatibilidad suma cdot con h} are satisfied if and only if conditions~\eqref{condicion equivalente compatibilidad suma cdot con y} and~\eqref{condicion equivalente compatibilidad suma cdot con h} are.
\begin{comment} prueba de lo afirmado
\begin{align*}
(y  + w_h & + y'+w_{h'})\cdot y'' = (y^{\beta}_{y,y',h,h'} + w_{h+h'})\cdot y'' = ((h+h')\blackdiamond y^{\beta}_{y,y',h,h'})\cdot ((h+h')\blackdiamond y'')
\end{align*}
\begin{align*}
((y+ w_h)\cdot (y'+w_{h'}))\cdot ((y+ w_h)\cdot y'') = \bigl(y^f_{y,y',h,h'} + w_{h\cdot h'}\bigr)\cdot \bigl((h\blackdiamond y)\cdot (h\blackdiamond y'')\bigr)
\end{align*}
\begin{align*}
(y  + w_h & + y'+w_{h'})\cdot w_{h''} = (y^{\beta}_{y,y',h,h'} +w_{h+h'})\cdot w_{h''} \\
%
& =	((h+h')\blackdiamond y^{\beta}_{y,y',h,h'})\cdot f(h+h',h'') +  (h+h')\blackdiamond (y^{\beta}_{y,y',h'h})\Yleft (h+h')\cdot h'' + w_{(h+h')\cdot h''}
\end{align*}
and
\begin{align*}
((y+ w_h) &\cdot (y'+w_{h'}))\cdot ((y+ w_h)\cdot w_{h''}) =   (y^f_{y,y',h,h'} + w_{h\cdot h'})\cdot (y^f_{y,0,h,h''} + w_{h\cdot h''})\\
%
& = ((h\cdot h')\blackdiamond y^f_{y,y',h,h'})\cdot ((h\cdot h')\blackdiamond y^f_{y,0,h,h''})
+ ((h\cdot h')\blackdiamond y^f_{y,y',h,h'})\cdot f(h\cdot h',h\cdot h'')\\
%
&+ (h\cdot h')\blackdiamond y^f_{y,y',h,h'}\Yleft (h\cdot h')\cdot(h\cdot h'') + w_{(h\cdot h')\cdot(h\cdot h'')}.
\end{align*}
\end{comment}
\end{proof}

\begin{remark}\label{caso particular de condicion equivalente compatibilidad suma cdot con h} Identity~\eqref{condicion equivalente compatibilidad suma cdot con h} with $h=h'=0$ becomes $(y+y')\Yleft h'' = (y\cdot y')\cdot (y\Yleft h'') + y\cdot y' \Yleft h''.$
\end{remark}

Recall here that we are assuming that $y\cdot w_h = y\Yleft h + w_h$ and  $w_h\cdot y = h\blackdiamond y$ for all $h\in H$ and $y\in I$.

\begin{proposition}\label{caracterizacion de extensiones} Under the conditions at the beginning of this section,  the abelian group $I\times_{\beta} H$ is a linear cycle set via \eqref{formula para cdot en I times H} and \eqref{extension} is an extension of linear cycle sets  if and only if the maps $y\mapsto h\blackdiamond y$ are permutations of $I$ and identities~\eqref{segunda equivalencia de (2.1)1}, \eqref{segunda equivalencia de (2.1)2}, \eqref{condicion equivalente compatibilidad suma cdot con y} and~\eqref{condicion equivalente compatibilidad suma cdot con h} are satisfied.
\end{proposition}

\begin{proof} This is an immediate consequence of Remark~\ref{son biyectivas} and Propositions~\ref{segunda equivalencia de (2.1)} and~\ref{equivalencia de (a+b).c = (a.b).(a.c)}.
\end{proof}

\begin{notation}\label{notacion para la estructura} The linear cycle set with underlying abelian group $I\times_{\beta} H$, constructed in the previous proposition, will be denoted $I\times_{\beta,f}^{\blackdiamond,\Yleft} H$. We will use the same notation for $I\times H$, endowed with the operations~$+$ and~$\cdot$ defined by~\eqref{formula para suma en I times H} and~\eqref{formula para cdot en I times H}, also when $I\times_{\beta,f}^{\blackdiamond,\Yleft} H$ is not a linear cycle set.
\end{notation}

\begin{remark}\label{toda extension es equivalente a una...} Let $I$ and $H$ be linear cycle sets, let $(\iota,B,\pi)$ be an extension of $H$ by $I$ and let $s$ be a set theoretic section of $\pi$, such that $s(0)=0$. Let $\beta$, $\blackdiamond$, $\Yleft$ and $f$ be the maps constructed in~\eqref{construccion de beta} and subsection~\ref{formula for cdot}. Consider the short exact sequence of abelian groups~\eqref{extension}, with $I\times_{\beta} H$ endowed with the operation $\cdot$, introduced above Proposition~\ref{segundas propiedades}. The map $\phi\colon I\times_{\beta} H\to B$, defined by $\phi(y+w_h)\coloneqq \iota(y)+s(h)$ is a bijection compatible with $+$ and $\cdot$. Hence~\eqref{extension} is an extension of linear cycle sets and $\phi$ is an equivalence of extensions.
\end{remark}

\begin{proposition}\label{equivalecia entre extensiones producto} Let $I$ and $H$ be linear cycle sets. Two extensions $(\iota,I\times_{\beta,f}^{\blackdiamond,\Yleft} H,\pi)$ and $(\iota,I\times_{\beta',f'}^{\blackdiamond',\Yleft'} H,\pi)$, of $H$ by $I$, are equivalent if and only if there exists a map $\varphi\colon H\to I$ satisfying~\eqref{equivalencia a nivel de qrupos} and the following conditions
\begin{align}
& h\blackdiamond \varphi(h)\in \Soc(I)\quad\text{for all $h\in H$,}\label{primera condicion}\\
& \varphi(h) = y\cdot \varphi(h)\quad\text{ for all $y\in I$ and $h\in H$,}\label{segunda condicion}\\
&\varphi(h\cdot h')= h\blackdiamond \varphi(h') + h\blackdiamond \varphi(h)\Yleft h\cdot h'\quad\text{for all $h,h'\in H$.}\label{tercera condicion}
\end{align}
In this case the map $\phi\colon I\times_{\beta,f}^{\blackdiamond,\Yleft} H\to I\times_{\beta',f'}^{\blackdiamond',\Yleft'} H$ performing the equivalence is given by $\phi(y+w_h)= y+\varphi(h) + w_h$.
\end{proposition}

\begin{proof} By the discussion above Definition~\ref{extensiones de linear cycle sets} we know that $(\iota,I\times_{\beta,f}^{\blackdiamond, \Yleft} H, \pi)$ and $(\iota,I\times_{\beta',f'}^{\blackdiamond',\Yleft'} H,\pi)$ are equivalent as abelian group extensions if and only if there exists a map $\varphi\colon H\to I$ satisfying~\eqref{equivalencia a nivel de qrupos}. Moreover the map $\phi\colon I\times_{\beta,f}^{\blackdiamond,\Yleft} H\to I\times_{\beta',f'}^{\blackdiamond',\Yleft'} H$, defined by $\phi(y+w_h)\coloneqq y+\varphi(h) + w_h$, performs the equivalence. Since
$$
\phi\bigl((y+w_h)\cdot y'\bigr)= (h\blackdiamond y)\cdot (h\blackdiamond y')\quad\text{and}\quad \phi(y+w_h)\cdot \phi(y') = ((h\blackdiamond y)\cdot (h\blackdiamond\varphi(h)))\cdot ((h\blackdiamond y)\cdot (h\blackdiamond y')),
$$	
identity $\phi\bigl((y+w_h)\cdot y'\bigr)= \phi(y+w_h)\cdot \phi(y')$ is fulfilled if and only if
$$
((h\blackdiamond y)\cdot (h\blackdiamond\varphi(h)))\cdot ((h\blackdiamond y)\cdot (h\blackdiamond y')) = (h\blackdiamond y)\cdot (h\blackdiamond y') \quad\text{for all $y,y'\in I$ and $h\in H$,}
$$	
or in other words, if and only if $(h\blackdiamond y)\cdot (h\blackdiamond\varphi(h))\in \Soc(I)$ for all $h\in H$ and $y\in I$. But this occurs if and only if $h\blackdiamond\varphi(h)\in \Soc(H)$ for all $h\in H$, because  $(h\blackdiamond 0)\cdot (h\blackdiamond\varphi(h)) = h\blackdiamond\varphi(h)$ and $\Soc(I)$ is an ideal. Assume now that condition~\eqref{primera condicion} is satisfied. A direct computation shows that
\begin{align*}
&\phi\bigl((y+w_h)\cdot w_{h'}\bigr) = (h\blackdiamond y) \cdot f(h,h') + h\blackdiamond y 	\Yleft h\cdot h' +\varphi(h\cdot h') + w_{h\cdot h'}
\shortintertext{and}
& \phi(y+w_h)\cdot \phi(w_{h'}) = (h\blackdiamond (y+\varphi(h)))\cdot (h\blackdiamond \varphi(h') +  f(h,h')) + h\blackdiamond (y+\varphi(h)) \Yleft h\cdot h' + w_{h\cdot h'}\\
& \phantom{\phi(y+w_h)\cdot \phi(w_{h'})} =  (h\blackdiamond y+ h\blackdiamond \varphi(h))\cdot (h\blackdiamond \varphi(h') +  f(h,h')) + (h\blackdiamond y+h\blackdiamond \varphi(h)) \Yleft h\cdot h' + w_{h\cdot h'}\\
& \phantom{\phi(y+w_h)\cdot \phi(w_{h'})} = (h\blackdiamond y)\cdot (h\blackdiamond \varphi(h')) + (h\blackdiamond y)\cdot f(h,h') + (h\blackdiamond y)\cdot (h\blackdiamond \varphi(h) \Yleft h\cdot h') + h\blackdiamond y \Yleft h\cdot h'  + w_{h\cdot h'},
\end{align*}
where for the last equality we have used~\eqref{compatibilidad cdot suma} and~\eqref{compatibilidad suma cdot}, that $h\blackdiamond \varphi(h)\in \Soc(I)$ and Remark~\ref{caso particular de condicion equivalente compatibilidad suma cdot con h}. Therefore, the ident\-ity $\phi\bigl((y+w_h)\cdot w_{h'}\bigr)= \phi(y+w_h)\cdot \phi(w_{h'})$ is fulfilled if and only if
\begin{equation}\label{segunda y tercera}
\varphi(h\cdot h')=(h\blackdiamond y)\cdot (h\blackdiamond \varphi(h')) + (h\blackdiamond y)\cdot (h\blackdiamond \varphi(h)\Yleft h\cdot h') \quad\text{for all $y\in I$ and $h,h'\in H$.}
\end{equation}
Taking $h=0$ in this equality we obtain condition~\eqref{segunda condicion}; while taking $y=0$, we obtain
condition~\eqref{tercera condicion}.~Re\-cip\-rocally, these conditions imply equality~\eqref{segunda y tercera}.
\end{proof}	

\begin{remark}\label{Yleft= Yleft' y blackdiamond=blackdiamond'} Let $\phi\colon I\times_{\beta,f}^{\blackdiamond,\Yleft} H\to I \times_{\beta',f'}^{\blackdiamond',\Yleft'} H$ be as in Proposition~\ref{equivalecia entre extensiones producto} and let $y\in I$ and $h\in H$. By the definition of $\phi$, and identities~\eqref{compatibilidad cdot suma}, \eqref{compatibilidad suma cdot}, ~\eqref{primera condicion} and~\eqref{segunda condicion}, we have
\begin{align*}
& y\Yleft h=\phi(y\Yleft h)= \phi(y\cdot w_h - w_h)= \phi(y)\cdot \phi(w_h) - \phi(w_h)= y\cdot (\varphi(h)+w_h) - (\varphi(h)+w_h)=y\Yleft' h\\
\shortintertext{and}
& h\blackdiamond y =\phi(h\blackdiamond y) = \phi(w_h)\cdot \phi(y)= (w_h + \varphi(h)) \cdot y=(w_h \cdot \varphi(h)) \cdot (w_h\cdot y) = (h\blackdiamond' \varphi(h))\cdot (h\blackdiamond' y)=h\blackdiamond' y,
\end{align*}
Thus, $\Yleft = \Yleft'$ and $\blackdiamond=\blackdiamond'$.
\end{remark}

%\begin{remark}\label{Yleft= Yleft' y blackdiamond=blackdiamond'}
%By Proposition~\ref{equivalecia entre extensiones producto}, the fact that $\phi$ is a linear cycle set isomorphism that it is the identity on $I$,  Remark~\ref{varphi cae en el semicentro} and identities~\eqref{compatibilidad cdot suma} and~\eqref{compatibilidad suma cdot}, we have
%$$
%y\Yleft h=\phi(y\Yleft h)= \phi(y\cdot w_h - w_h)= \phi(y)\cdot \phi(w_h) - \phi(w_h)= y\cdot (\varphi(h)+w_h) - (\varphi(h)+w_h)=y\Yleft' h
%$$
%and
%$$
%h\blackdiamond y =\phi(h\blackdiamond y) = \phi(w_h)\cdot \phi(y)= (w_h + \varphi(h)) \cdot y=(w_h \cdot \varphi(h)) \cdot (w_h\cdot y)= (h\blackdiamond \varphi(h))\cdot (h\blackdiamond y)=h\blackdiamond y.
%$$
%Thus, $\Yleft = \Yleft'$ and $\blackdiamond=\blackdiamond'$.
%\end{remark}

\section{Extensions with central cocycles}\label{Extensions with central cocycles}
In addition to the conditions at the beginning of Section~\ref{Building extensions of linear cycle sets}, throughout this section we assume that $h\blackdiamond (y\Yleft h')$, $h\blackdiamond \beta(h',h'')$ and $h\blackdiamond f(h',h'')$ belong to $Z(I)$ for all $h,h',h''\in H$ and $y\in I$.

\begin{remark}\label{equivalencia de [3.3]} Under the previous assumptions, equality~\eqref{segunda equivalencia de (2.1)2} becomes
\begin{equation}\label{segunda equivalencia de (2.1)2 cociclos en el centro}
y\Yleft (h'+h'') = y\Yleft h' + y\Yleft h''.
\end{equation}
\end{remark}

\begin{lemma}\label{condicion [3.6]} Assume that $I\times_{\beta} H$ satisfies condition~\eqref{compatibilidad cdot suma}. Then~\eqref{condicion equivalente compatibilidad suma cdot con y} holds if and only if
\begin{equation}\label{equivalente a [3.6]}
(h\blackdiamond y)\cdot (h\blackdiamond y')= h\blackdiamond (y\cdot y')\qquad\text{and}\qquad (h+h')\blackdiamond y = (h\cdot h') \blackdiamond (h\blackdiamond y)
\end{equation}
for all $h,h'\in H$ and $y,y'\in I$.
\end{lemma}

\begin{proof} A direct computation shows that, if we are under the assumptions at the beginning of this subsection, then condition~\eqref{condicion equivalente compatibilidad suma cdot con y} reduces to the fact that, for all $h,h'\in H$ and $y,y',y''\in I$,
$$
\bigl((h+h')\blackdiamond (y+y')\bigr) \cdot \bigl((h+h')\blackdiamond y''\bigr)=\bigl((h\cdot h')\blackdiamond ((h\blackdiamond y)\cdot (h\blackdiamond y'))\bigr)\cdot \bigl((h\cdot h')\blackdiamond ((h\blackdiamond y)\cdot (h\blackdiamond y''))\bigr).
$$ 	
%\begin{comment}
%Prueba de lo afirmado arriba
\noindent This follows since
\begin{align*}
((h+h') \blackdiamond & y^{\beta}_{y,y',h,h'})\cdot ((h+h')\blackdiamond y'') = \bigl((h+h')\blackdiamond (y+y')+(h+h')\blackdiamond \beta(h,h')\bigr) \cdot \bigl((h+h')\blackdiamond y''\bigr)\\
& = \bigl(((h+h')\blackdiamond \beta(h,h'))\cdot ((h+h')\blackdiamond (y+y'))\bigr) \cdot \bigl(((h+h')\blackdiamond \beta(h,h'))\cdot((h+h')\blackdiamond y'')\bigr)\\
& = ((h+h')\blackdiamond (y+y'))\cdot((h+h')\blackdiamond y'')
\end{align*}
and, similarly,
\begin{align*}
&((h\cdot h')\blackdiamond y^f_{y,y',h,h'})\cdot ((h\cdot h')\blackdiamond ((h\blackdiamond y)\cdot (h\blackdiamond y'')))\\
&((h\cdot h')\blackdiamond ((h\blackdiamond y)\cdot (h\blackdiamond y') + (h\blackdiamond y) \cdot f(h,h') + h\blackdiamond y \Yleft h\cdot h'
))\cdot ((h\cdot h')\blackdiamond ((h\blackdiamond y)\cdot (h\blackdiamond y'')))\\
&=\bigl((h\cdot h')\blackdiamond ((h\blackdiamond y)\cdot (h\blackdiamond y')) + (h\cdot h')\blackdiamond \cdot f(h,h') + (h\cdot h')\blackdiamond (h\blackdiamond y \Yleft h\cdot h')\bigr)\cdot ((h\cdot h')\blackdiamond ((h\blackdiamond y)\cdot (h\blackdiamond y'')))\\
& = \bigl((h\cdot h')\blackdiamond ((h\blackdiamond y)\cdot (h\blackdiamond y'))\bigr)\cdot \bigl((h\cdot h')\blackdiamond ((h\blackdiamond y)\cdot (h\blackdiamond y''))\bigr).
\end{align*}
%\end{comment}
Taking $y'=0$ and $h=0$ in this equality, we obtain the first equality in~\eqref{equivalente a [3.6]}; while taken $y=y'=0$, we obtain the second one. On the other hand, if the equalities in~\eqref{equivalente a [3.6]} are true, then
\begin{align*}
\bigl((h+h')\blackdiamond (y+y')\bigr) \cdot \bigl((h+h')\blackdiamond y''\bigr) &= \bigl(h+h'\bigr)\blackdiamond \bigl((y+y')\cdot y''\bigr)\\
&=\bigl(h\cdot h'\bigr)\blackdiamond \bigl(h\blackdiamond ((y+y')\cdot y'')\bigr)\\
&=\bigl(h\cdot h'\bigr)\blackdiamond \bigl(h\blackdiamond ((y\cdot y')\cdot (y\cdot y''))\bigr)\\
%
%&=\bigl(h\cdot h'\bigr)\blackdiamond \bigl((h\blackdiamond (y\cdot y'))\cdot (h\blackdiamond (y\cdot y''))\bigr)\\
%
%&= \bigl((h\cdot h')\blackdiamond ((h\blackdiamond (y\cdot y'))\bigr)\cdot \bigl((h\cdot h')\blackdiamond ((h\blackdiamond (y\cdot y''))\bigr)\\
%
&=\bigl((h\cdot h')\blackdiamond ((h\blackdiamond y)\cdot (h\blackdiamond y'))\bigr)\cdot \bigl((h\cdot h')\blackdiamond ((h\blackdiamond y)\cdot (h\blackdiamond y''))\bigr),
\end{align*}
as desired.
\end{proof}

\begin{remark}\label{condicion equivalente a la segunda igualdad en la proposicion anterior} Since the maps $h'\mapsto h\cdot h'$ are permutations of $H$, from the first equality in~\eqref{equivalencia linear cycle set braza} it follows that the second equality in~\eqref{equivalente a [3.6]} holds if and only if
\begin{equation}\label{asociatividad de diamante negro nueva}
hh'\blackdiamond y = h'\blackdiamond (h\blackdiamond y)\quad\text{for all $h,h'\in H$ and $y\in I$.}
\end{equation}
\end{remark}

\begin{remark}\label{equivalencia de [3.7]} Under the hypotheses assumed in this subsection, \eqref{condicion equivalente compatibilidad suma cdot con h} becomes
\begin{equation}
\begin{aligned}\label{condicion equivalente compatibilidad suma cdot con hsimp nueva}
f(h+h',h'') + (h+h')\blackdiamond y^{\beta}_{y,y',h,h'}\Yleft (h+h')\cdot h''&= (h\cdot h')\blackdiamond y^f_{y,0,h,h''}+ f(h\cdot h',h\cdot h'') \\
&  + (h\cdot h')\blackdiamond y^f_{y,y',h,h'}\Yleft (h\cdot h')\cdot(h\cdot h'').
\end{aligned}
\end{equation}
\end{remark}

\begin{lemma}\label{condicion [3.7]} Suppose that $I\times_{\beta} H$ satisfies condition~\eqref{compatibilidad cdot suma} and that identities~\eqref{equivalente a [3.6]} hold. Then formula~\eqref{condicion equivalente compatibilidad suma cdot con hsimp nueva} holds if and only if the following conditions are fulfilled:
\begin{align}
& (y+y')\Yleft h = y\Yleft h + y\cdot y' \Yleft h,\label{condicion equivalente compatibilidad suma cdot con hmassimp1 nuevo}\\
& h\blackdiamond y \Yleft h\cdot h' = h\blackdiamond (y\Yleft h') + h\blackdiamond (y\Yleft h) \Yleft h\cdot h',\label{condicion equivalente compatibilidad suma cdot con hmassimp2 nuevo}\\
\shortintertext{and}
&\begin{aligned}\label{condicion equivalente compatibilidad suma cdot con hmassimp3 nuevo}
f(h+h',h'') + (h+h')\blackdiamond \beta(h,h') \Yleft (h+h')\cdot h'' & = (h\cdot h')\blackdiamond f(h,h'') + f(h\cdot h',h\cdot h'')\\
&+ (h\cdot h')\blackdiamond f(h,h') \Yleft (h\cdot h')\cdot(h\cdot h'').
\end{aligned}
\end{align}
\end{lemma}

\begin{proof} When $h=h'=0$, formula~\eqref{condicion equivalente compatibilidad suma cdot con hsimp nueva} becomes identity~\eqref{condicion equivalente compatibilidad suma cdot con hmassimp1 nuevo}; while, when $h=0$ and $y'=0$, formula~\eqref{condicion equivalente compatibilidad suma cdot con hsimp nueva}~be\-comes identity~\eqref{condicion equivalente compatibilidad suma cdot con hmassimp2 nuevo}. Thus, in order to finish the proof will be sufficient to show that, if these identities are satisfied, then conditions~\eqref{condicion equivalente compatibilidad suma cdot con hsimp nueva} and~\eqref{condicion equivalente compatibilidad suma cdot con hmassimp3 nuevo} are equivalent. A direct computation using~\eqref{condicion equivalente compatibilidad suma cdot con hmassimp1 nuevo}, proves that
\begin{align*}
& (h+h')\blackdiamond y^{\beta}_{y,y',h,h'}\Yleft (h+h')\cdot h'' = (h+h')\blackdiamond \beta(h,h') \Yleft (h+h')\cdot h'' + (h+h')\blackdiamond (y+y') \Yleft (h+h')\cdot h'',\\
&(h\cdot h') \blackdiamond y^f_{y,0,h,h''} = (h\cdot h')\blackdiamond f(h,h'') + (h\cdot h')\blackdiamond (h\blackdiamond y \Yleft h\cdot h''),\\
& (h\cdot h') \blackdiamond y^f_{y,y',h,h'}\Yleft (h\cdot h')\cdot(h\cdot h'') = (h\cdot h')\blackdiamond f(h,h') \Yleft (h\cdot h')\cdot(h\cdot h'')\\
&\phantom{(h\cdot h') \blackdiamond y^f_{y,y',h,h'}\Yleft (h\cdot h')\cdot(h\cdot h'')} + (h\cdot h')\blackdiamond \bigl((h\blackdiamond y)\cdot (h\blackdiamond y')\bigr) \Yleft (h\cdot h') \cdot (h\cdot h'')\\
&\phantom{(h\cdot h') \blackdiamond y^f_{y,y',h,h'}\Yleft (h\cdot h')\cdot(h\cdot h'')} + (h\cdot h')\blackdiamond (h\blackdiamond y \Yleft h\cdot h') \Yleft (h\cdot h')\cdot(h\cdot h'').
\end{align*}
Thus, we are reduced to check that
\begin{align*}
(h+h')\blackdiamond (y+y') \Yleft (h+h')\cdot h'' &= (h\cdot h')\blackdiamond (h\blackdiamond y \Yleft h\cdot h'') + (h\cdot h')\blackdiamond \bigl((h\blackdiamond y)\cdot (h\blackdiamond y')\bigr) \Yleft (h\cdot h')\cdot(h\cdot h'')\\
& + (h\cdot h')\blackdiamond (h\blackdiamond y \Yleft h\cdot h') \Yleft (h\cdot h')\cdot(h\cdot h'').
\end{align*}
But, since by identities~\eqref{compatibilidad suma cdot}, \eqref{equivalente a [3.6]} and~\eqref{condicion equivalente compatibilidad suma cdot con hmassimp1 nuevo},
\begin{align*}
(h+h')& \blackdiamond (y+y') \Yleft (h+h')\cdot h'' = \bigl((h+h')\blackdiamond y+(h+h')\blackdiamond y'\bigr) \Yleft (h+h')\cdot h'' \\
& = (h+h')\blackdiamond y \Yleft (h+h')\cdot h'' + ((h+h')\blackdiamond y)\cdot ((h+h')\blackdiamond y')\Yleft (h+h')\cdot h''\\
& = (h\cdot h') \blackdiamond (h\blackdiamond y) \Yleft (h\cdot h')\cdot(h\cdot h'') + ((h\cdot h')\blackdiamond (h\blackdiamond y))\cdot ((h\cdot h') \blackdiamond (h\blackdiamond y'))\Yleft (h\cdot h')\cdot(h\cdot h'')\\
& = (h\cdot h') \blackdiamond (h\blackdiamond y) \Yleft (h\cdot h')\cdot(h\cdot h'') + (h\cdot h')\blackdiamond \bigl((h\blackdiamond y)\cdot (h\blackdiamond y')\bigr) \Yleft (h\cdot h')\cdot(h\cdot h''),
\end{align*}
for this we only must show that
$$
(h\cdot h') \blackdiamond (h\blackdiamond y) \Yleft (h\cdot h')\cdot(h\cdot h'') = (h\cdot h')\blackdiamond (h\blackdiamond y \Yleft h\cdot h'') + (h\cdot h')\blackdiamond (h\blackdiamond y \Yleft h\cdot h') \Yleft (h\cdot h')\cdot(h\cdot h''),
$$
which follows immediately from equality~\eqref{condicion equivalente compatibilidad suma cdot con hmassimp2 nuevo}.
\end{proof}

Recall again that we are assuming that $y\cdot w_h = y\Yleft h + w_h$ and  $w_h\cdot y = h\blackdiamond y$ for all $h\in H$ and $y\in I$.

\begin{theorem}\label{condiciones para ser extension, cociclos en el centro} Under the conditions at the beginning ot the section, the abelian group $I\times_{\beta} H$ is a linear cycle set via \eqref{formula para cdot en I times H} and the short exact sequence sequence~\eqref{extension} is an extension of linear cycle sets if and only if the maps $y\mapsto h\blackdiamond y$ are permutations of $I$ and con\-di\-tions~\eqref{segunda equivalencia de (2.1)1}, \eqref{segunda equivalencia de (2.1)2 cociclos en el centro}, \eqref{equivalente a [3.6]}, \eqref{condicion equivalente compatibilidad suma cdot con hmassimp1 nuevo}, \eqref{condicion equivalente compatibilidad suma cdot con hmassimp2 nuevo} and~\eqref{condicion equivalente compatibilidad suma cdot con hmassimp3 nuevo} are satisfied.
\end{theorem}

\begin{proof} This is a consequence of Proposition~~\ref{caracterizacion de extensiones}, Remarks~\ref{equivalencia de [3.3]} and~\ref{equivalencia de [3.7]}, and Lemmas~\ref{condicion [3.6]} and~\ref{condicion [3.7]}. 	
\end{proof}

In the following result we do not assume a priori that the hypotheses at the beginning of this section are satisfied.

\begin{corollary}\label{ppal seccion 4}
Let $I$ and $H$ be linear cycle sets and let $\blackdiamond\colon H\times I\to I$, $\Yleft \colon I\times H\to I$ and $f\colon H\times H\to I$  be maps satisfying conditions~(1)--(4) of Proposition~\ref{primeras propiedades}. Assume that $h\blackdiamond (y\Yleft h')$, $h\blackdiamond \beta(h',h'')$ and $h\blackdiamond f(h',h'')$ belong to $Z(I)$ for all $h,h',h''\in H$ and $y\in I$, and let $I\times_{\beta,f}^{\blackdiamond,\Yleft} H$ be as in Notation~\ref{notacion para la estructura}. Then $I\times_{\beta,f}^{\blackdiamond,\Yleft} H$ is a linear cycle set and the short exact sequence sequence~\eqref{extension} is an extension of linear cycle sets if and only if the maps $y\mapsto h\blackdiamond y$ are permutations of $I$ and
\begin{align*}
& \beta(h,h')+\beta(h+h',h'') = \beta(h',h'')+\beta(h,h'+h''),\\
& \beta(h,0) = \beta(0,h) = 0\quad\text{and} \quad \beta(h,h')= \beta(h',h),\\
& h\blackdiamond  \beta(h',h'') + f(h,h'+h'') = f(h,h')+f(h,h'')+\beta(h\cdot h',h\cdot h''),\\
& y\Yleft (h'+h'') = y\Yleft h' + y\Yleft h'',\\
&(h\blackdiamond y)\cdot (h\blackdiamond y')= h\blackdiamond (y\cdot y')\quad\text{and}\quad (h+h')\blackdiamond y = (h\cdot h') \blackdiamond (h\blackdiamond y),\\
& (y+y')\Yleft h = y\Yleft h + y\cdot y' \Yleft h,\\
& h\blackdiamond y \Yleft h\cdot h' = h\blackdiamond (y\Yleft h') + h\blackdiamond (y\Yleft h) \Yleft h\cdot h'
\shortintertext{and}
&\begin{aligned}\label{condicion equivalente compatibilidad suma cdot con hmassimp3 nuevo}
f(h+h',h'') + (h+h')\blackdiamond \beta(h,h') \Yleft (h+h')\cdot h'' & = (h\cdot h')\blackdiamond f(h,h'') + f(h\cdot h',h\cdot h'')\\
&+ (h\cdot h')\blackdiamond f(h,h') \Yleft (h\cdot h')\cdot(h\cdot h''),
\end{aligned}
\end{align*}
for all $h,h',h''\in H$ and $y,y'\in I$.
\end{corollary}

\begin{proof}
Since the displayed formulas are conditions~\eqref{condicion de cociclo}, \eqref{normalidad y cociclo abeliano}, \eqref{segunda equivalencia de (2.1)1}, \eqref{segunda equivalencia de (2.1)2 cociclos en el centro}, \eqref{equivalente a [3.6]}, \eqref{condicion equivalente compatibilidad suma cdot con hmassimp1 nuevo}, \eqref{condicion equivalente compatibilidad suma cdot con hmassimp2 nuevo} and~\eqref{condicion equivalente compatibilidad suma cdot con hmassimp3 nuevo}, the result follows immediately from the comments at the beginning of Section~\ref{Extensions of linear cycle sets} and Theorem~\ref{condiciones para ser extension, cociclos en el centro}.
\end{proof}

\begin{proposition}\label{en terminos de triangulo} Let $\triangleleft\colon I \times H\to I$ be the map defined by $y\triangleleft h \coloneqq  h\blackdiamond (y - y\Yleft h)$. If identities~\eqref{segunda equivalencia de (2.1)2 cociclos en el centro}, \eqref{asociatividad de diamante negro nueva} and~\eqref{condicion equivalente compatibilidad suma cdot con hmassimp1 nuevo} hold and the maps $y\mapsto h \blackdiamond y$ are permutations of $I$, then condition~\eqref{condicion equivalente compatibilidad suma cdot con hmassimp2 nuevo} is fulfilled if and only if
\begin{equation}\label{asosiatividad de triangulo a izquierda nuevo}
y\triangleleft hh'= (y\triangleleft h)\triangleleft h'\quad\text{for all $h,h'\in H$ and $y\in I$.}
\end{equation}
\end{proposition}

\begin{proof} On one hand, by \eqref{asociatividad de diamante negro nueva}, we have
$$
y\triangleleft hh'= hh'\blackdiamond (y-y\Yleft hh')= h'\blackdiamond \bigl(h\blackdiamond y - h\blackdiamond (y\Yleft hh')\bigr)\quad \text{for all $y\in I$ and $h,h'\in H$}.
$$
On the other hand, a direct computation using~\eqref{condicion equivalente compatibilidad suma cdot con hmassimp1 nuevo}, shows that, also for all $y\in I$ and $h,h'\in H$,
\begin{align*}
(y\triangleleft h)\triangleleft h' &= h'\blackdiamond \bigl(h\blackdiamond (y-y\Yleft h) - h\blackdiamond (y-y\Yleft h)\Yleft h'\bigr)\\
& = h'\blackdiamond \bigl(h\blackdiamond y- h\blackdiamond (y\Yleft h) - h\blackdiamond y\Yleft h' + h \blackdiamond (y\Yleft h)\Yleft h'\bigr).
\end{align*}
Thus, since the maps $y\mapsto h \blackdiamond y$ are permutations of $I$, using~\eqref{segunda equivalencia de (2.1)2 cociclos en el centro} we get that equality~\eqref{asosiatividad de triangulo a izquierda nuevo} holds if and only if
$$
h\blackdiamond y\Yleft h' = h\blackdiamond (y\Yleft (hh'-h)) + h \blackdiamond (y\Yleft h)\Yleft h' \quad\text{for all  $h,h'\in H$ and $y\in I$.}
$$
But it is clear that this equality is true if and only if~\eqref{condicion equivalente compatibilidad suma cdot con hmassimp2 nuevo} is (take into account the first equality in~\eqref{equivalencia linear cycle set braza}).
\end{proof}

\begin{remark}\label{caso I trivial} When $I$ is trivial, the first condition in~\eqref{equivalente a [3.6]} is automatically fulfilled and identity~\eqref{condicion equivalente compatibilidad suma cdot con hmassimp1 nuevo}~be\-comes $(y+y')\Yleft h = y\Yleft h +  y' \Yleft h$.
\end{remark}

\begin{remark}\label{socalo y centro son cerrados por diamante} Let $(\iota,I\times_{\beta,f}^{\blackdiamond, \Yleft} H, \pi)$ be an extension of linear cycle sets and let $y\in I$ and $h\in H$. Since the~maps $y'\mapsto h'\blackdiamond y'$ are permutations of $I$,  from the first iden\-tity in~\eqref{equivalente a [3.6]}, it follows that if $y\in \Soc(I)$, then $h\blackdiamond y\in \Soc(I)$, and if $y\in \Z(I)$, then $h\blackdiamond y\in \Z(I)$.
\end{remark}

%\begin{remark}\label{socalo y centro son cerrados por diamante} Let $(\iota,I\times_{\beta,f}^{\blackdiamond, \Yleft} H, \pi)$ be an extension of linear cycle sets and let $y\in I$ and $h\in H$. By the first~iden\-tity in~\eqref{equivalente a [3.6]}, if $y\in \Soc(I)$, then $h\blackdiamond y\in \Soc(I)$; while if $y\in \Z(I)$, then $h\blackdiamond y\in \Z(I)$.
%\end{remark}

\begin{proposition}\label{vrphi es central} Let $I$ and $H$ be linear cycle sets and let $(\iota,I\times_{\beta,f}^{\blackdiamond,\Yleft} H,\pi)$ and $(\iota,I\times_{\beta',f'}^{\blackdiamond,\Yleft} H,\pi)$ be equivalent extensions of $H$ by $I$. The map $\varphi\colon H\to I$ of Proposition~\ref{equivalecia entre extensiones producto} takes its values in $\Z(I)$.
\end{proposition}

\begin{proof} By~\eqref{compatibilidad suma cdot} and the fact that $h\blackdiamond(\varphi(h)\Yleft h)\in\Soc(I)$, we have
$$
h\blackdiamond y  = \phi(w_h\cdot y) = \phi(w_h)\cdot\phi(y)=(\varphi(h)+w_h)\cdot y =(\varphi(h)\cdot w_h)\cdot (\varphi(h)\cdot y)=(\varphi(h)\Yleft h + w_h)\cdot (\varphi(h)\cdot y)= h\blackdiamond (\varphi(h)\cdot y).
$$
Since the map $y\mapsto h\blackdiamond y$ is bijective, $\varphi(h)\in\Soc(I)$ for all $h\in H$. By condition~\eqref{segunda condicion} this finishes the proof.
\end{proof}

\subsection[Extensions of linear cycle sets by trivial ideals]{Extensions of linear cycle sets by trivial ideals}
In this subsection we assume that $I$ is trivial and set $y^h\coloneqq h^{-1}\blackdiamond (y \triangleleft h)$ for each $y\in I$ and $h\in H$.

\begin{remark}\label{otra forma} Items (1)--(4) of Proposition~\ref{primeras propiedades} and conditions~\eqref{segunda equivalencia de (2.1)2 cociclos en el centro}, \eqref{equivalente a [3.6]}, \eqref{condicion equivalente compatibilidad suma cdot con hmassimp1 nuevo} and \eqref{condicion equivalente compatibilidad suma cdot con hmassimp2 nuevo} say that
\begin{align}
& y \triangleleft hh' = (y \triangleleft h) \triangleleft h',\quad (y + y')\triangleleft h  = y\triangleleft h + y'\triangleleft h,\quad y \triangleleft 0 = y,\label{eqq 1}\\
& h'h\blackdiamond y  = h\blackdiamond (h' \blackdiamond y),\quad h\blackdiamond (y + y') = h\blackdiamond y + h\blackdiamond y',\quad 0\blackdiamond y = y \label{eqq 2}\\
& y^{h+h'} + y = y^h + y^{h'}\quad \text{and}\quad 0^h = 0, \label{eqq 3}
\end{align}
for all $h,h'\in H$ and $y,y'\in I$. More precisely,
	
\begin{itemize}[itemsep=0.0ex, topsep=1.0ex, label=-]
		
\item The first condition in \eqref{equivalente a [3.6]} is automatically fulfilled, since $I$ is trivial; while the second one is equivalent to~\eqref{asociatividad de diamante negro nueva}, which is the first condition in~\eqref{eqq 2}.
		
\item Items~(1) and~(2) of Proposition~\ref{primeras propiedades} are the second and third conditions in~\eqref{eqq 2}; and item~(4) of~Propo\-sition~\ref{primeras propiedades} is equivalent to the last condition in~\eqref{eqq 1}. Since the first and third conditions in~\eqref{eqq 2} imply that $y^h=y-y\Yleft h$, item~(3) of Proposition~\ref{primeras propiedades} is equivalent to the last condition in~\eqref{eqq 3}.
		
\item Condition~\eqref{segunda equivalencia de (2.1)2 cociclos en el centro} is satisfied if an only if $y^{h+h'} + y = y^h + y^{h'}$.
		
\item Condition~\eqref{condicion equivalente compatibilidad suma cdot con hmassimp1 nuevo} (combined with Proposition~\ref{primeras propiedades}(1)) says that $(y + y')\triangleleft h  = y\triangleleft h + y'\triangleleft h$.
		
\item Finally, by Proposition~\ref{en terminos de triangulo} we know that condition~\eqref{condicion equivalente compatibilidad suma cdot con hmassimp2 nuevo} says that $y \triangleleft hh' = (y \triangleleft h) \triangleleft h'$.
	
\end{itemize}
Note that, by~\eqref{eqq 1} and~\eqref{eqq 2},
\begin{equation}
y^0 = y\qquad\text{and}\qquad (y_1+y_2)^h = y_1^h + y_2^h.\label{eq 3'}
\end{equation}
\end{remark}

Assume that~\eqref{extension} is a short exact sequences of linear cycle sets. Let $h\in H$ and $y\in I$.  A direct computation using  $y^{-1}=-y$, $y\cdot w_h=y\Yleft h + w_h$, $w_h\cdot w_h =-w_h^{-1}$ and Remark~\ref{caso I trivial}, proves that
$$
y\Yleft h + y^{-1} \Yleft h=0\Yleft h=0\quad\text{and}\quad w_h\cdot (y^{-1}\cdot w_h) = h\blackdiamond (y^{-1}\Yleft h) - w_h^{-1}.
$$
Using that $y\triangleleft h =  h\blackdiamond (y - y\Yleft h)$, these facts and the second condition in~\eqref{equivalencia linear cycle set braza}, we obtain that
$$
y\triangleleft h =h\blackdiamond(y^{-1} \Yleft h+y) = w_h\cdot (y^{-1}\cdot w_h)+ w_h^{-1}+h\blackdiamond y= w_h\cdot (y^{-1}\cdot w_h)+ w_h^{-1} + w_h\cdot y= w_h^{-1}yw_h.
$$
So, the map $\triangleleft$ is the right action $\sigma$ introduced by Bachiller above~\cite{B}*{Definition 3.1}.
\begin{comment}
%Lo anterior suponiendo solo que los cociclos son centrales
A direct computation using that $y\cdot w_h=y\Yleft h + w_h$, $y\Yleft h$ belongs to the center of $I$,  $y^{-1}\cdot w_h=y^{-1}\Yleft h + w_h$ and $w_h\cdot w_h =-w_h^{-1}$, proves that
$$
y\Yleft h + y^{-1} \Yleft h=0\Yleft h=0
w_h=y^{-1}\cdot (y\cdot w_h)= y\Yleft h + y^{-1} \Yleft h + w_h \quad\text{and}\quad w_h\cdot (y^{-1}\cdot w_h) = h\blackdiamond (y^{-1}\Yleft h) - w_h^{-1}.
$$
Using these facts we obtain that
$$
y\triangleleft h =  h\blackdiamond (y - y\Yleft h)=h\blackdiamond(y^{-1} \Yleft h+ y) = w_h\cdot (y^{-1}\cdot w_h)+ w_h^{-1}+h\blackdiamond y= w_h\cdot (y^{-1}\cdot w_h)+ w_h^{-1}+w_h\cdot y= w_h^{-1}yw_h.
$$
So, the map $\triangleleft$ defined in Proposition~\ref{en terminos de triangulo} is the right action $\sigma$ introduced in~\cite{Ba}. On the other hand
\end{comment}
On the other hand, the map $\nu$, also introduced above~\cite{B}*{Definition 3.1}, is given by $\nu_h(y)=(w_h)^{-1}\cdot y=w_{h^{-1}}\cdot y=h^{-1} \blackdiamond y$.

\begin{remark}\label{condicion para I incluido en Soc} Since $I$ is trivial, $I\subseteq \Soc\bigl(I\times_{\beta,f}^{\blackdiamond,\Yleft} H\bigr)$ if and only if $y\Yleft h= y\cdot w_h-w_h = 0$ for all $y\in I$ and $h\in H$. Since $y^h=y-y\Yleft h$, this happens if and only if $y^h=y$ for all $y\in I$ and $h\in H$. Moreover, in this case condition~\eqref{condicion equivalente compatibilidad suma cdot con hmassimp3 nuevo} becomes
\begin{equation}\label{condicion equivalente compatibilidad suma cdot con hmassimp3 caso simple}
f(h+h',h'')  = (h\cdot h')\blackdiamond f(h,h'') + f(h\cdot h',h\cdot h'').
\end{equation}
\end{remark}

\begin{remark}\label{equivalecia entre extensiones producto, caso I trivial} Let $I$ and $H$ be linear cycle sets, with $I$ trivial. By Proposition~\ref{equivalecia entre extensiones producto}, two extensions $(\iota,I\times_{\beta,f}^{\blackdiamond,\Yleft} H,\pi)$ and $(\iota,I\times_{\beta',f'}^{\blackdiamond,\Yleft} H,\pi)$, of $H$ by $I$, are equivalent if and only if there exists a map $\varphi\colon H\to I$ satisfying~\eqref{equivalencia a nivel de qrupos} and
\begin{equation}\label{tercera condicion caso I trivial}
\varphi(h\cdot h')= h\blackdiamond \varphi(h') + h\blackdiamond \varphi(h)\Yleft h\cdot h'\quad\text{for all $h,h'\in H$.}
\end{equation}
Note that, when $\Yleft=0$, identity \eqref{tercera condicion caso I trivial} becomes $\varphi(h\cdot h')= h\blackdiamond \varphi(h')$.
\end{remark}

\section[Cohomology of linear cycle sets (case \texorpdfstring{$\Yleft=0$}{<=0})]{Cohomology of linear cycle sets (case \texorpdfstring{$\pmb{\Yleft=0}$}{<=0})}\label{Cohomology of linear cycle sets1}
Let $H$ be a linear cycle set. For each $s\ge 1$, let $\mathds{Z}[\ov{H}^s]$ be the free abelian group with basis $\ov{H}^s$, where $\ov{H}\coloneqq H\setminus \{0\}$. We let $\sh(\mathds{Z}[\ov{H}^s])$ denote the subgroup of $\mathds{Z}[\ov{H}^s]$ generated by the shuffles
$$
\sum_{\sigma\in \sh_{l,s-l}} \sg(\sigma)(h_{\sigma^{-1}(1)},\dots, h_{\sigma^{-1}(s)}),
$$
taken for all $1\le l < s$ and $h_k\in\ov{H}$. Here $\sh_{l,s-l}$ is the subset of all the permutations $\sigma$ of $s$ elements satisfying $\sigma(1)<\cdots<\sigma(l)$ and $\sigma(l+1)<\cdots<\sigma(s)$. For each $r\ge 0$ and $s\ge 1$, let $\wh{C}^N_{rs}(H)\coloneqq \mathds{Z}[\ov{H}^{r}]\ot \ov{M}(s)$, where $\ov{M}(s)\coloneqq \frac{\mathds{Z}[\ov{H}^{s}]}{\sh(\mathds{Z}[\ov{H}^{s})]}$. Given $h_1,\dots,h_s\in \ov{H}$, we let $[h_1,\dots,h_s]$ denote the class of $(h_1,\dots,h_s)$ in $\ov{M}(s)$. Let $I$ be an abelian group endowed with a map $\blackdiamond \colon H\times I \to I$ satisfying conditions~\eqref{eqq 2}. Let $\wh{C}_N^{rs}(H,I)\coloneqq \Hom(\wh{C}^N_{rs}(H),I)$. Notice that $\wh{C}_N^{rs}(H,I)$ is isomorphic to the subgroup of $\Hom(\mathds{Z}[\ov{H}^{r+s}], I)$ consisting of all the maps $f$ such that
$$
\sum_{\sigma\in \sh_{l,s-l}} \sg(\sigma) f(h_1,\dots,h_r, h_{r+\sigma^{-1}(1)},\dots, h_{r+\sigma^{-1}(s)})=0,
$$
for all $h_1,\dots,h_{r+s}\in \ov{H}$ and $1\le l<s$. For all $r\in \mathds{N}$, let $R_{r-1,1}\colon \mathds{Z}[\ov{H}^r]\to I$ be the map  defined by
$$
R_{r-1,1}(h_1,\dots,h_r)\coloneqq (h_1+\cdots + h_{r-1})\cdot h_r.
$$
Consider the diagram $\bigl(\wh{C}_N^{**}(H,I),\partial_{\mathrm{h}}^{**},\partial_{\mathrm{v}}^{**}\bigr)$, where
$$
\partial_{\mathrm{h}}^{r+1,s}\colon \wh{C}_N^{rs}(H,I)\to \wh{C}_N^{r+1,s}(H,I)\quad\text{and}\quad \partial_{\mathrm{v}}^{r,s+1}\colon \wh{C}_N^{rs}(H,I)\to \wh{C}_N^{r,s+1}(H,I)
$$
are the maps defined by
\begin{align*}
& \partial_{\mathrm{h}}^{r+1,s}(f)(h_1,\dots, h_{r+1},[h_{r+2},\dots,h_{r+s+1}]) \coloneqq f(h_1\cdot h_2,\dots,h_1\cdot h_{r+1},[h_1\cdot h_{r+2},\dots,h_1\cdot h_{r+s+1}])\\
&\phantom{\partial_{\mathrm{h}}^{r+1,s}(f)(h_1,\dots,h_{r+1}} + \sum_{\jmath=1}^{r} (-1)^{\jmath} f(h_1,\dots, h_{\jmath-1},h_{\jmath}+h_{\jmath+1}, h_{\jmath+2},\dots,h_{r+1},[h_{r+2},\dots, h_{r+s+1}])\\
&\phantom{\partial_{\mathrm{h}}^{r+1,s}(f)(h_1,\dots,h_{r+1}} + (-1)^{r+1} R_{r1}(h_1,\dots,h_{r+1})\blackdiamond f(h_1,\dots,h_r,[h_{r+2},\dots, h_{r+s+1}])
\shortintertext{and}
& \partial_{\mathrm{v}}^{r,s+1}(f)(h_1,\dots,h_r,[h_{r+1},\dots,h_{r+s+1}]) \coloneqq (-1)^r f(h_1,\dots,h_r,[h_{r+2},\dots,h_{r+s+1}])\\
&\phantom{\partial_{\mathrm{v}}^{r,s+1}(f)(h_1,\dots,h_r} + \sum_{\jmath=r+1}^{r+s} (-1)^{\jmath} f(h_1,\dots,h_r,[h_{r+1},\dots,h_{\jmath-1}, h_{\jmath}+h_{\jmath+1},h_{\jmath+2},\dots, h_{r+s+1}])\\
&\phantom{\partial_{\mathrm{v}}^{r,s+1}(f)(h_1,\dots,h_r} + (-1)^{r+s+1} f(h_1,\dots,h_r,[h_{r+1},\dots,h_{r+s}]).
\end{align*}

\begin{theorem}\label{principal} The diagram $\bigl(\wh{C}_N^{**}(H,I),\partial_{\mathrm{h}}^{**},\partial_{\mathrm{v}}^{**}\bigr)$ is a double cochain complex.
\end{theorem}

\begin{proof} We must prove that $\partial_{\mathrm{v}}^{r,s+1}\xcirc \partial_{\mathrm{v}}^{rs} = 0$, $\partial_{\mathrm{h}}^{r+1,s}\xcirc \partial_{\mathrm{h}}^{rs} = 0$ and $\partial_{\mathrm{v}}^{r+1,s+1}\xcirc \partial_{\mathrm{h}}^{rs} = - \partial_{\mathrm{h}}^{r+1,s+1}\xcirc \partial_{\mathrm{v}}^{rs}$. The first equality is well known, while the third one follows by a direct computation using identity~\eqref{compatibilidad cdot suma}. Next we prove the second one. For $r,s\in \mathds{N}$ and $0\le\jmath\le r$, let $d^{\jmath}_{rs}\colon \mathds{Z}[H^r]\ot \overline{M}(s) \to \mathds{Z}[H^{r-1}]\ot \overline{M}(s)$ be the maps defined by
\begin{align*}
d^0_{rs}(\ov{h}) &\coloneqq (h_1\cdot h_2,\dots,h_1\cdot h_r)\ot  [h_1\cdot h_{r+1},\dots,h_1\cdot h_{r+s}], \\
d^{\jmath}_{rs}(\ov{h}) &\coloneqq (h_1,\dots h_{\jmath-1},h_{\jmath}+h_{\jmath+1},h_{\jmath+2},\dots, h_r)\ot [h_{r+1},\dots, h_{r+s}]\quad\text{for $j=1,\dots,r-1$,}\\
d^r_{rs}(\ov{h}) &\coloneqq (h_1,\dots,h_{r-1})\ot [h_{r+1},\dots,h_{r+s}].
\end{align*}
where $\ov{h}\coloneqq (h_1,\dots,h_r)\ot [h_{r+1},\dots,h_{r+s}]$. Then,
$$
\partial_{\mathrm{h}}^{rs}(f)(\ov{h})=f(d^0_{rs}(\ov{h}))+\sum_{\jmath=1}^{r-1} (-1)^{\jmath}f(d^{\jmath}_{rs}(\ov{h}))+(-1)^r R_{r-1,1}(h_1,\dots,h_r)\blackdiamond f(d^r_{rs}(\ov{h})).
$$
Let $\ov{h}\coloneqq (h_1,\dots,h_{r+1})\ot [h_{r+2},\dots,h_{r+s+1}]$ and $\mathrm{h}_{1l}\coloneqq (h_1,\dots,h_l)$. Note that
\begin{align*}
\partial_{\mathrm{h}}^{r+1,s}(\partial_{\mathrm{h}}^{rs}(f))(\ov{h}) & = \sum_{\imath=0}^{r-1} \sum_{\jmath=0}^r (-1)^{\imath+\jmath} f(d^{\imath}_{rs}(d^{\jmath}_{r+1,s}(\ov{h})))\\
& + \sum_{\imath=0}^{r-1} (-1)^{r+\imath-1} R_{r1}(\mathrm{h}_{1,r+1})\blackdiamond f(d^{\imath}_{rs}(d^{r+1}_{r+1,s}(\ov{h})))\\
& - R_{r1}(\mathrm{h}_{1,r+1})\blackdiamond \bigl(R_{r-1,1}(\mathrm{h}_{1r})\blackdiamond f(d^r_{rs}(d^{r+1}_{r+1,s}(\ov{h})))\bigr)\\
&+\sum_{\jmath=0}^r (-1)^{\jmath+r} R_{r-1,1}\left(d^{\jmath}_{r+1}(\mathrm{h}_{1,r+1})\right)\blackdiamond f(d^r_{rs}( d^{\jmath}_{r+1,s}(\ov{h}))),
\end{align*}
where
$$
d^0_{r+1}(\mathrm{h}_{1,r+1}) = (h_1\cdot h_2,\dots,h_1\cdot h_{r+1})\quad\text{and}\quad d^{\jmath}_{r+1}(\mathrm{h}_{1,r+1})= (h_1,\dots h_{\jmath-1},h_{\jmath}+h_{\jmath+1},h_{\jmath+2},\dots,h_{r+1}),
$$
for $j=1,\dots,r$. Hence, in order to check that $\partial_{\mathrm{h}}^{r+1,s}\xcirc \partial_{\mathrm{h}}^{rs} = 0$, it suffices to prove that
\begin{align}
&\sum_{\imath=0}^{r-1} \sum_{\jmath=0}^r (-1)^{\imath+\jmath} f(d^{\imath}_{rs}(d^{\jmath}_{r+1,s}(\ov{h})))=0,\label{ecua1}\\
&\sum_{\imath=0}^{r-1} (-1)^{\imath}R_{r1}(\mathrm{h}_{1,r+1})\blackdiamond f(d^{\imath}_{rs}(d^{r+1}_{r+1,s}(\ov{h}))) + \sum_{\jmath=0}^{r-1} (-1)^{\jmath+1} R_{r-1,1}\left(d^{\jmath}_{r+1}(\mathrm{h}_{1,r+1})\right)\blackdiamond f(d^r_{rs}(d^{\jmath}_{r+1,s}(\ov{h})))=0,\label{ecua2}\\
&R_{r1}(\mathrm{h}_{1,r+1})\blackdiamond \bigl(R_{r-1,1}(\mathrm{h}_{1r})\blackdiamond f(d^r_{rs}(d^{r+1}_{r+1,s}(\ov{h})))\bigr) - R_{r-1,1}\left(d^r_{r+1}(\mathrm{h}_{1,r+1})\right)\blackdiamond f(d^r_{rs}(d^{\jmath}_{r+1,s}(\ov{h})))=0.\label{ecua3}
\end{align}
Equality~\eqref{ecua1} follows immediately from the fact that, by the associativity of $+$ and identity~\eqref{compatibilidad cdot suma},
\begin{equation}\label{simplicial}
d^{\imath}_{rs}\xcirc d^{\jmath}_{r+1,s}=d^{\jmath-1}_{r+1,s}\xcirc d^{\imath}_{rs}\qquad\text{for $0\le \imath<\jmath\le r$.}
\end{equation}
In order to prove~\eqref{ecua2} it suffices to verify that
\begin{equation}\label{ecua4}
d^{\jmath}_{rs}\xcirc d^{r+1}_{r+1,s}=d^r_{rs}\xcirc d^{\jmath}_{r+1,s}\quad\text{and}\quad R_{r1} = R_{r-1,1}\xcirc d^{\jmath}_{r+1} \qquad\text{for $j=0,\dots,r-1$.}
\end{equation}
The first equalities are clear. We prove the second equalities by induction on $r$. The case $r=1$ follows from the fact that $R_{11} = R_{01}\xcirc d^0_2$ by~\eqref{compatibilidad cdot suma} and~\eqref{compatibilidad suma cdot}. Assume that the result is true for $r$. Then
$$
R_{r+1,1} = R_{r1}\xcirc d^0_{r+2} = R_{r-1,1}\xcirc d^{\jmath}_{r+1}\xcirc d^0_{r+2} = R_{r-1,1}\xcirc d^0_{r+2}\xcirc d^{\jmath+1}_{r+1} = R_{r1}\xcirc d^{\jmath+1}_{r+1},
$$
which proves the case $\jmath>0$ for $r+1$ (the case $\jmath = 0$ holds by~\eqref{compatibilidad cdot suma} and~\eqref{compatibilidad suma cdot}).
	
In order to check equality~\eqref{ecua3} it suffices to show that
$$
d^r_{rs}\xcirc d^{r+1}_{r+1,s}=d^r_{rs}\xcirc d^r_{r+1,s}\quad\text{and}\quad R_{r-1,1}(\mathrm{h}_{1r})R_{r1}(\mathrm{h}_{1,r+1}) = R_{r-1,1}\left(d^r_{r+1}(\mathrm{h}_{1,r+1}) \right).
$$
The first equality is clearly true; while the second one follows immediately from~\eqref{formula que usaremos}.
\end{proof}

\begin{notation} We let $\ho_{\blackdiamond}^*(H,I)$ denote the cohomology groups of $\bigl(\wh{C}_N^{**}(H,I), \partial_{\mathrm{h}}^{**}, \partial_{\mathrm{v}}^{**}\bigr)$.
\end{notation}

\begin{remark}
When $\blackdiamond$ is the trivial action, the double cochain complex $\bigl(\wh{C}_N^{**}(H,I),\partial_{\mathrm{h}}^{**}, \partial_{\mathrm{v}}^{**}\bigr)$ is canonically isomorphic to the double normalized cochain complex introduced in~\cite{LV}*{Section 4}.
\end{remark}

Assume now that $I$ is a trivial linear cycle set endowed with a map $\blackdiamond\colon H\times I\to H$ satisfying conditions~\eqref{eqq 2}. Set $y\Yleft h\coloneqq 0$. Note that under these assumptions conditions (1)--(4) of Proposition~\ref{primeras propiedades} are fulfilled and the maps $y\mapsto h\blackdiamond y$ are permutations of $I$. Moreover, conditions~\eqref{segunda equivalencia de (2.1)2 cociclos en el centro}, \eqref{equivalente a [3.6]}, \eqref{condicion equivalente compatibilidad suma cdot con hmassimp1 nuevo} and \eqref{condicion equivalente compatibilidad suma cdot con hmassimp2 nuevo} are also fulfilled.

\begin{proposition}\label{significado de cociclos} A pair $(\beta,f)\in \wh{C}_N^{02}(H,I)\oplus \wh{C}_N^{11}(H,I)$ is a cocycle of $\Tot\bigl(\wh{C}_N^{**}(H,I),\partial_{\mathrm{h}}^{**},\partial_{\mathrm{v}}^{**}\bigr)$ if and only if conditions~\eqref{condicion de cociclo},~\eqref{normalidad y cociclo abeliano},~\eqref{segunda equivalencia de (2.1)1}  and~\eqref{condicion equivalente compatibilidad suma cdot con hmassimp3 nuevo} are satisfied.
\end{proposition}

\begin{proof} This follows by a direct computation (take into account~\eqref{condicion equivalente compatibilidad suma cdot con hmassimp3 caso simple}).
\end{proof}

\begin{corollary}\label{equivalencia} Under the assumptions above Proposition~\ref{significado de cociclos}, there is a bijective correspondence
$$
\Ext_{\blackdiamond}(H,I) \longleftrightarrow \ho_{\blackdiamond}^2(H,I).
$$

\begin{proof}
By the discussion at the beginning of Section~\ref{Extensions of linear cycle sets}, Theorem~\ref{condiciones para ser extension, cociclos en el centro} and the discussion above Proposition~\ref{significado de cociclos}, we know that \eqref{extension} is an extension of linear cycle sets if and only if~\eqref{condicion de cociclo},~\eqref{normalidad y cociclo abeliano},~\eqref{segunda equivalencia de (2.1)1} and~\eqref{condicion equivalente compatibilidad suma cdot con hmassimp3 nuevo} are fulfilled, which by Proposition~\ref{significado de cociclos} means that $(\beta, f)$ is a cocycle of $\Tot\bigl(\wh{C}_N^{**}(H,I),\partial_{\mathrm{h}}^{**},\partial_{\mathrm{v}}^{**}\bigr)$. Now the result follows immediately from the fact that
$$
\partial_{\mathrm{v}}^{02}(\varphi)(h_1,h_2)= - \varphi(h_2)+ \varphi(h_1+h_2) - \varphi(h_1)\qquad \text{and}\qquad   \partial_{\mathrm{h}}^{11}(\varphi)(h_1,h_2)= \varphi(h_1\cdot h_2) - h_1\blackdiamond \varphi(h_2),
$$
and Remark~\ref{equivalecia entre extensiones producto, caso I trivial}.
\end{proof}
\end{corollary}

\section{Cohomology of linear cycle sets}\label{Cohomology of linear cycle sets}
Let $H$ and $I$ be linear cycle sets and $\blackdiamond\colon H\times I\to I$ a map. Assume that $I$ is trivial, that conditions~(1) and~(2) of Proposition~\ref{primeras propiedades} are fulfilled and that the maps $y\mapsto h\blackdiamond y$ are permutations. In the previous section we introduced a cohomology theory $\ho_{\blackdiamond}^*(H,I)$ such that $\ho_{\blackdiamond}^2(H,I)$ classify the equivalence classes of extensions $(\iota,B,\pi)$, of $H$ by $I$, with $\iota(I)\subseteq \Soc(B)$. The groups $\ho_{\blackdiamond}^*(H,I)$ are the cohomology groups of the double complex $\bigl(\wh{C}_N^{**}(H,I),\partial_{\mathrm{h}}^{**},\partial_{\mathrm{v}}^{**}\bigr)$ introduced in Theorem~\ref{principal}. In this section we will modify the cohomology theory, in order to classify extensions of $H$ by $I$, that not necessarily satisfy $\iota(I)\subseteq \Soc(B)$. For this, we fix $H$, $I$ and $\blackdiamond$ as above, and a map $\Yleft \colon I\times H\to I$, and we assume that conditions~\eqref{eqq 1}--\eqref{eqq 3} are fulfilled (by Remark~\ref{otra forma} this means that items (1)--(4) of Proposition~\ref{primeras propiedades} and conditions~\eqref{segunda equivalencia de (2.1)2 cociclos en el centro}, \eqref{equivalente a [3.6]}, \eqref{condicion equivalente compatibilidad suma cdot con hmassimp1 nuevo} and \eqref{condicion equivalente compatibilidad suma cdot con hmassimp2 nuevo} are satisfied). We are going to define diagonal maps $D_{rs}^{r+s,1}$, from $\wh{C}_N^{rs}(H,I)$ to $\wh{C}_N^{r+s,1}(H,I)$, that yield a degree~$1$ endomorphism $D$ of $\Tot\bigl(\wh{C}_N^{**}(H,I)\bigl)$ such that $\bigl(\Tot(\wh{C}_N^{**}(H,I)),\partial+D\bigr)$ is a cochain complex (where $\partial\coloneqq \partial_h+\partial_v$), and we are going to prove the second cohomology group $\ho_{\blackdiamond,\Yleft}^2(H,I)$ of $\bigl(\Tot(\wh{C}_N^{**}(H,I)),\partial+D\bigr)$ is in bijection with $\Ext_{\blackdiamond,\Yleft}(H;I)$. When $\Yleft =0$, then $D=0$ and the results obtained in this section coincide with the ones obtained in Section~\ref{Cohomology of linear cycle sets1}.

\smallskip

\begin{notation} We let $\wh{C}_N^n(H,I)$ denote the $n$th group of $\Tot(\wh{C}_N^{**}(H,I))$. That is,
$$
\wh{C}_N^n(H,I)\coloneqq \bigoplus_{r=0}^{n-1}\wh{C}_N^{r,n-r}(H,I).
$$
\end{notation}

\begin{notation} For the sake of brevity we introduce the following notations:

\begin{itemize}

\item[-] Given $h_1,\dots,h_{r+s}$ we set $\mathrm{h}_{1,r+s} = h_1,\dots,h_{r+s}\coloneqq (h_1,\dots,h_r)\ot [h_{r+1},\dots,h_{r+s}]$.

\item[-] Given $h_1,\dots,h_r\in \ov{H}$ and $1\le i\le j\le r$, we write $\mathrm{h}_{ij}\coloneqq h_i,h_{i+1},\dots,h_j$.

\item[-] Given $h,h_1,\dots,h_r\in \ov{H}$, we set $h\cdot \mathrm{h}_{1r}\coloneqq h\cdot h_1,\dots,h\cdot h_{r+1}$.

\item[-] For $g\colon \wh{C}^N_{rs}(H)\to H$ and $f\in \wh{C}_N^{rs}(H,I)$, we set $(g\blackdiamond f) (\mathrm{h}_{1,r+s})\coloneqq g(\mathrm{h}_{1,r+s})\blackdiamond  f(\mathrm{h}_{1,r+s})$.

\end{itemize}
\end{notation}

\noindent For all $r\ge0$ and $s\ge 1$, we define $R_{rs}\colon \wh{C}^N_{rs}(H)\to H$, by
$$
R_{rs}(\mathrm{h}_{1,r+s})\coloneqq (h_1+\dots+h_r)\cdot (h_{r+1}+\dots+h_{r+s}).
$$
Then, for all $r\ge0$ and $s\ge 1$, we define $D_{rs}^{r+s,1}\colon \wh{C}_N^{rs}(H,I)\to \wh{C}_N^{r+s,1}(H,I)$,by
$$
D_{rs}^{r+s,1}(f)(\mathrm{h}_{1,r+s+1})=(-1)^{r+s} (R_{rs}\blackdiamond f)(\mathrm{h}_{1,r+s}) \Yleft R_{r+s,1}(\mathrm{h}_{1,r+s+1}).
$$
Finally, for $n\ge 1$ we let $D_n^{n+1}\colon \wh{C}_N^n(H,I)\to \wh{C}_N^{n+1}(H,I)$ denote the map, given by
$$
D_n^{n+1}(f_{0n},\dots,f_{n-1,1})\coloneqq \sum_{k=1}^n D_{n-k,k}^{n1}(f_{n-k,k}).
$$

Figure~\ref{ddiagonales} at the end of this section shows in black the double cochain complex $\bigl(\wh{C}_N^{**}(H,I),\partial_{\mathrm{h}}^{**},\partial_{\mathrm{v}}^{**}\bigr)$, and in green the maps $D_{rs}^{r+s,1}$.

\smallskip

We next prove that $(\partial+D)^2=0$, in other words, that $\bigl(\wh{C}_N^*(H,I),\partial+D\bigr)$ is a complex. For this we note that the restriction of $\partial+D$ to the summand $\wh{C}_N^{rs}(H,I)$ of $\wh{C}_N^{r+s}(H,I))$ is the sum $\partial_{\mathrm{h}}^{r+1,s} + \partial_{\mathrm{v}}^{r,s+1} + D_{rs}^{r+s,1}$. Therefore we must compute the sum of the nine compositions in the right hand of
\begin{equation}\label{nueve terminos}
(\partial_{\mathrm{h}}+\partial_{\mathrm{v}}+D)\xcirc (\partial_{\mathrm{h}}+\partial_{\mathrm{v}}+D)=\partial_{\mathrm{h}}\xcirc \partial_{\mathrm{h}}+\partial_{\mathrm{v}}\xcirc \partial_{\mathrm{h}}+D\xcirc \partial_{\mathrm{h}} +
\partial_{\mathrm{h}}\xcirc \partial_{\mathrm{v}}+\partial_{\mathrm{v}}\xcirc \partial_{\mathrm{v}}+D\xcirc \partial_{\mathrm{v}} +\partial_{\mathrm{h}}\xcirc D+\partial_{\mathrm{v}}\xcirc D+D\xcirc D.
\end{equation}

\begin{lemma}\label{lema partial v D} $\partial_{\mathrm{v}}^{r2}\xcirc D_{r-k,k}^{r1}=0$, for all $r\ge k\ge 1$.
\end{lemma}

\begin{proof} Since, for all $\varphi\in \wh{C}_N^{r1}(H,I)$,
$$
\partial_{\mathrm{v}}^{r2}(\varphi)(\mathrm{h}_{1,r+2})=(-1)^r\varphi(\mathrm{h}_{1r},h_{r+2}) +(-1)^{r+1}\varphi(\mathrm{h}_{1r},h_{r+1}+h_{r+2}) + (-1)^r\varphi(\mathrm{h}_{1,r+1}),
$$
for all $f\in \wh{C}_N^{r-k,k}(H,I)$, we have
\begin{align*}
\partial_{\mathrm{v}}^{r2}\xcirc D_{r-k,k}^{r1}(f)(\mathrm{h}_{1,r+2}])&=(-1)^r D_{r-k,k}^{r,1}(f)(\mathrm{h}_{1r},h_{r+2})\\
& +(-1)^{r+1}D_{r-k,k}^{r1}(f)(\mathrm{h}_{1r},h_{r+1}+h_{r+2})\\
& +(-1)^r D_{r-k,k}^{r1}(f)(\mathrm{h}_{1,r+1})\\
& = (R_{r-k,k}\blackdiamond f)(\mathrm{h}_{1r})\Yleft R_{r1}(\mathrm{h}_{1r},h_{r+2})\\
& - (R_{r-k,k}\blackdiamond f)(\mathrm{h}_{1r})\Yleft R_{r1}(\mathrm{h}_{1r},h_{r+1}+h_{r+2})\\
& + (R_{r-k,k}\blackdiamond f)(\mathrm{h}_{1r})\Yleft R_{r1}(\mathrm{h}_{1,r+1})\\
& =0,
\end{align*}
where the last equality is true by identity~\eqref{segunda equivalencia de (2.1)2 cociclos en el centro} and the fact that
\begin{align*}
R_{r1}(\mathrm{h}_{1r},h_{r+1}+h_{r+2}) & = (h_1+\cdots + h_r)\cdot (h_{r+1}+h_{r+2})\\
& = (h_1+\cdots+h_r)\cdot h_{r+1}+(h_1+\cdots+h_r)\cdot h_{r+2}\\
&= R_{r1}(\mathrm{h}_{1,r+1})+R_{r1}(\mathrm{h}_{1r},h_{r+2}),
\end{align*}
in which the second equality holds by~\eqref{compatibilidad cdot suma}.
\end{proof}

\begin{lemma}\label{lema partial h D} For each $\beta\in \wh{C}_N^{r-k,k}(H,I)$, the following equality holds:
\begin{align*}
(\partial_{\mathrm{h}}^{r+1,1} + D_{r1}^{r+1,1}) &\xcirc D_{r-k,k}^{r1}(\beta)(\mathrm{h}_{1,r+2}) = (-1)^r R_{r-k+1,k}(\mathrm{h}_{1,r+1}) \blackdiamond \beta(h_1\cdot \mathrm{h}_{2,r+1}) \Yleft R_{r+1,1}(\mathrm{h}_{1,r+2})\\
& + \sum_{\jmath=1}^{r-k} (-1)^{r+\jmath} R_{r-k+1,k}(\mathrm{h}_{1,r+1}) \blackdiamond \beta(\mathrm{h}_{1,\jmath-1},h_{\jmath}+h_{\jmath+1},\mathrm{h}_{\jmath+2,r+1})\Yleft R_{r+1,1}(\mathrm{h}_{1,r+2})\\
& + \sum_{\jmath=r-k+1}^{r} (-1)^{r+\jmath}R_{r-k,k+1}(\mathrm{h}_{1,r+1})\blackdiamond \beta(\mathrm{h}_{1,\jmath-1},h_{\jmath}+h_{\jmath+1},\mathrm{h}_{\jmath+2,r+1})\Yleft R_{r+1,1}(\mathrm{h}_{1,r+2})\\
& - R_{r-k,k+1}(\mathrm{h}_{1,r+1}) \blackdiamond \beta(\mathrm{h}_{1r})\Yleft R_{r+1,1}(\mathrm{h}_{1,r+2}).
\end{align*}
\end{lemma}

\begin{proof} By definition,
\begin{align*}
(\partial_{\mathrm{h}}^{r+1,1} + D_{r1}^{r+1,1})\xcirc D_{r-k,k}^{r1}(\beta)(\mathrm{h}_{1,r+2}) &= D_{r-k,k}^{r1}(\beta)(h_1\cdot \mathrm{h}_{2,r+2})\\
& + \sum_{\jmath=1}^r (-1)^{\jmath} D_{r-k,k}^{r1}(\beta)(\mathrm{h}_{1,\jmath-1},h_{\jmath}+h_{\jmath+1},\mathrm{h}_{\jmath+2,r+2})\\
& + (-1)^{r+1} R_{r1}(\mathrm{h}_{1,r+1}) \blackdiamond D_{r-k,k}^{r1}(\beta)(\mathrm{h}_{1r},h_{r+2})\\
& + (-1)^{r+1} (R_{r1}\blackdiamond D_{r-k,k}^{r1}(\beta))(\mathrm{h}_{1,r+1})\Yleft R_{r+1,1}(\mathrm{h}_{1,r+2})\\
& = (-1)^r (R_{r-k,k}\blackdiamond \beta)(h_1\cdot \mathrm{h}_{2,r+1}) \Yleft R_{r1}(h_1\cdot \mathrm{h}_{2,r+2})\\
& + \sum_{\jmath=1}^r (-1)^{r+\jmath}(R_{r-k,k} \blackdiamond \beta) (\mathrm{h}_{1,\jmath-1},h_{\jmath}+h_{\jmath+1},\mathrm{h}_{\jmath+2,r+1})\Yleft R_{r+1,1}(\mathrm{h}_{1,r+2})\\
& - R_{r1}(\mathrm{h}_{1,r+1}) \blackdiamond \bigl((R_{r-k,k} \blackdiamond \beta)(\mathrm{h}_{1r}) \Yleft R_{r1}(\mathrm{h}_{1r},h_{r+2})\bigr)\\
& - R_{r1}(\mathrm{h}_{1,r+1}) \blackdiamond \bigl((R_{r-k,k} \blackdiamond \beta)(\mathrm{h}_{1r}) \Yleft R_{r1}(\mathrm{h}_{1,r+1})\bigr)\Yleft R_{r+1,1}(\mathrm{h}_{1,r+2}).
\end{align*}
Since, by~\eqref{compatibilidad cdot suma} and~\eqref{compatibilidad suma cdot},

\begin{itemize}

\item[-] $R_{r+1,1}(\mathrm{h}_{1,r+2}) = R_{r1}(\mathrm{h}_{1,r+1})\cdot R_{r1}(\mathrm{h}_{1r},h_{r+2})$,

\item[-] $R_{r-k,k}(h_1\cdot \mathrm{h}_{2,r+1})= R_{r-k+1,k}(\mathrm{h}_{1,r+1})$,

\item[-] $R_{r1}(h_1\cdot \mathrm{h}_{2,r+2})=R_{r+1,1}(\mathrm{h}_{1,r+2})$,

\end{itemize}
we obtain
\begin{align*}
(\partial_{\mathrm{h}}^{r+1,1} + D_{r1}^{r+1,1}) &\xcirc D_{r-k,k}^{r1}(\beta)(\mathrm{h}_{1,r+2}) = (-1)^r R_{r-k+1,k}(\mathrm{h}_{1,r+1}) \blackdiamond \beta(h_1\cdot \mathrm{h}_{2,r+1}) \Yleft R_{r+1,1}(\mathrm{h}_{1,r+2})\\
& + \sum_{\jmath=1}^{r-k} (-1)^{r+\jmath} R_{r-k+1,k}(\mathrm{h}_{1,r+1}) \blackdiamond \beta(\mathrm{h}_{1,\jmath-1},h_{\jmath}+h_{\jmath+1},\mathrm{h}_{\jmath+2,r+1})\Yleft R_{r+1,1}(\mathrm{h}_{1,r+2})\\
& + \sum_{\jmath=r-k+1}^{r} (-1)^{r+\jmath}R_{r-k,k+1}(\mathrm{h}_{1,r+1})\blackdiamond \beta(\mathrm{h}_{1,\jmath-1},h_{\jmath}+h_{\jmath+1},\mathrm{h}_{\jmath+2,r+1})\Yleft R_{r+1,1}(\mathrm{h}_{1,r+2})\\
& - R_{r1}(\mathrm{h}_{1,r+1}) \blackdiamond \bigl((R_{r-k,k} \blackdiamond \beta)(\mathrm{h}_{1r}) \Yleft R_{r1}(\mathrm{h}_{1r},h_{r+2})\bigr)\\
& - R_{r1}(\mathrm{h}_{1,r+1}) \blackdiamond \bigl((R_{r-k,k} \blackdiamond \beta)(\mathrm{h}_{1r}) \Yleft R_{r1}(\mathrm{h}_{1,r+1})\bigr)\Yleft R_{r1}(\mathrm{h}_{1,r+1})\cdot R_{r1}(\mathrm{h}_{1r},h_{r+2}),
\end{align*}
where we additionally use that
$$
R_{r-k,k}(\mathrm{h}_{1,\jmath-1},h_{\jmath}+h_{\jmath+1},\mathrm{h}_{\jmath+2,r+1})=\begin{cases} R_{r-k+1,k}(\mathrm{h}_{1,r+1}), & \text{if $\jmath\le r-k$}\\ R_{r-k,k+1}(\mathrm{h}_{1,r+1}), & \text{if $r-k<\jmath\le  r$}\end{cases}.
$$
Using now~\eqref{condicion equivalente compatibilidad suma cdot con hmassimp2 nuevo}, we obtain that
\begin{align*}
(\partial_{\mathrm{h}}^{r+1,1} + D_{r1}^{r+1,1}) &\xcirc D_{r-k,k}^{r1}(\beta)(\mathrm{h}_{1,r+2}) = (-1)^r R_{r-k+1,k}(\mathrm{h}_{1,r+1}) \blackdiamond \beta(h_1\cdot \mathrm{h}_{2,r+1}) \Yleft R_{r+1,1}(\mathrm{h}_{1,r+2})\\
& + \sum_{\jmath=1}^{r-k} (-1)^{r+\jmath} R_{r-k+1,k}(\mathrm{h}_{1,r+1}) \blackdiamond \beta(\mathrm{h}_{1,\jmath-1},h_{\jmath}+h_{\jmath+1},\mathrm{h}_{\jmath+2,r+1})\Yleft R_{r+1,1}(\mathrm{h}_{1,r+2})\\
& + \sum_{\jmath=r-k+1}^{r} (-1)^{r+\jmath}R_{r-k,k+1}(\mathrm{h}_{1,r+1})\blackdiamond \beta(\mathrm{h}_{1,\jmath-1},h_{\jmath}+h_{\jmath+1},\mathrm{h}_{\jmath+2,r+1})\Yleft R_{r+1,1}(\mathrm{h}_{1,r+2})\\
& - R_{r1}(\mathrm{h}_{1,r+1})\blackdiamond \bigl((R_{r-k,k} \blackdiamond \beta)(\mathrm{h}_{1r})\bigr)\Yleft R_{r+1,1}(\mathrm{h}_{1,r+2}).
\end{align*}
From equality~\eqref{formula que usaremos} it follows that $R_{r-k,k}(\mathrm{h}_{1r})R_{r1}(\mathrm{h}_{1,r+1})=R_{r-k,k+1}(\mathrm{h}_{1,r+1})$. Using this and~\eqref{asociatividad de diamante negro nueva}, we conclude the proof.
\end{proof}

\begin{lemma}\label{lema D partial v} For $\beta\in C_{N}^{r-k,k}(H,I)$, we have
\begin{align*}
D_{r-k,k+1}^{r+1,1}&\xcirc \partial_{\mathrm{v}}^{r-k,k+1}(\beta)(\mathrm{h}_{1,r+2}) %= (-1)^{r+1} (R_{r-k,k+1}\blackdiamond \partial_{\mathrm{v}}^{r-k,k+1}(\beta)) (\mathrm{h}_{1,r+1})\Yleft R_{r+1,1}(\mathrm{h}_{1,r+2})\\
=(-1)^{k+1} R_{r-k,k+1}(\mathrm{h}_{1,r+1})\blackdiamond \beta(\mathrm{h}_{1,r-k},\mathrm{h}_{r-k+2,r+1})\Yleft R_{r+1,1}(\mathrm{h}_{1,r+2})\\
& + \sum_{\jmath=r-k+1}^r (-1)^{r+\jmath+1} R_{r-k,k+1}(\mathrm{h}_{1,r+1})\blackdiamond\beta(\mathrm{h}_{1,\jmath-1}, h_{\jmath}+h_{\jmath+1},\mathrm{h}_{\jmath+2,r+1})\Yleft R_{r+1,1}(\mathrm{h}_{1,r+2})\\
&  + R_{r-k,k+1}(\mathrm{h}_{1,r+1})\blackdiamond\beta(\mathrm{h}_{1r})\Yleft R_{r+1,1}(\mathrm{h}_{1,r+2}).
\end{align*}
\end{lemma}

\begin{proof} This follows by a direct computation.
\end{proof}

\begin{lemma}\label{lema D partial h} For $\beta\in \widehat{C}_N^{r-k,k}(H,I)$, we have
\begin{align*}
D_{r-k+1,k}^{r+1,1}&\xcirc \partial_{\mathrm{h}}^{r-k+1,k}(\beta)(\mathrm{h}_{1,r+2}) = (-1)^{r+1} (R_{r-k,k+1}\blackdiamond \partial_{\mathrm{h}}^{r-k,k+1}(\beta))(\mathrm{h}_{1,r+1})\Yleft R_{r+1,1}(\mathrm{h}_{1,r+2})\\
& =(-1)^{r+1} R_{r-k+1,k}(\mathrm{h}_{1,r+1})\blackdiamond \beta (h_1\cdot \mathrm{h}_{2,r+1})\Yleft R_{r+1,1}(\mathrm{h}_{1,r+2})\\
& +  \sum_{\jmath=1}^{r-k} (-1)^{r+\jmath+1} R_{r-k+1,k}(\mathrm{h}_{1,r+1})\blackdiamond \beta (\mathrm{h}_{1,\jmath-1}, h_{\jmath}+h_{\jmath+1},\mathrm{h}_{\jmath+2,r+1})\Yleft R_{r+1,1}(\mathrm{h}_{1,r+2})\\
&  + (-1)^k R_{r-k,k+1}(\mathrm{h}_{1,r+2})\blackdiamond \beta(\mathrm{h}_{1,r-k},\mathrm{h}_{r-k+2,r+1})\Yleft R_{r+1,1}(\mathrm{h}_{1,r+2}).
\end{align*}
\end{lemma}

\begin{proof} By definition
\begin{align*}
D_{r-k+1,k}^{r+1,1}&\xcirc \partial_{\mathrm{h}}^{r-k+1,k}(\beta)(\mathrm{h}_{1,r+2}) = (-1)^{r+1} (R_{r-k,k+1}\blackdiamond \partial_{\mathrm{h}}^{r-k,k+1}(\beta)) (\mathrm{h}_{1,r+1})\Yleft R_{r+1,1}(\mathrm{h}_{1,r+2})\\
& =(-1)^{r+1} R_{r-k+1,k}(\mathrm{h}_{1,r+1})\blackdiamond \beta (h_1\cdot \mathrm{h}_{2,r+1})\Yleft R_{r+1,1}(\mathrm{h}_{1,r+2})\\
& + \sum_{\jmath=1}^{r-k} (-1)^{r+\jmath+1} R_{r-k+1,k}(\mathrm{h}_{1,r+1})\blackdiamond \beta (\mathrm{h}_{1,\jmath-1}, h_{\jmath}+h_{\jmath+1},\mathrm{h}_{\jmath+2,r+1})\Yleft R_{r+1,1}(\mathrm{h}_{1,r+2})\\
&  + (-1)^k R_{r-k+1,k}(\mathrm{h}_{1,r+1})\blackdiamond \bigl(R_{r-k,1}(\mathrm{h}_{1,r-k+1})\blackdiamond \beta(\mathrm{h}_{1,r-k},\mathrm{h}_{r-k+2,r+1})\bigr)\Yleft R_{r+1,1}(\mathrm{h}_{1,r+2}).
\end{align*}
From equality~\eqref{formula que usaremos} it follows that $R_{r-k,1}(\mathrm{h}_{1,r-k+1}) R_{r-k+1,k}(\mathrm{h}_{1,r+1}) = R_{r-k,k+1}(\mathrm{h}_{1,r+1})$. Using this and~\eqref{asociatividad de diamante negro nueva}, we conclude the proof.
\end{proof}

\begin{theorem}\label{complejo caso general}
$(\wh{C}_N^*(H,I),\partial+D)$ is a cochain complex.
\end{theorem}

\begin{proof} We must prove that $(\partial+D)^2=0$ or, equivalently, the sum of the nine summands on the right side    of~\eqref{nueve terminos} vanishes. By Theorem~\ref{principal} we know that $\partial_{\mathrm{h}}\xcirc \partial_{\mathrm{h}}+\partial_{\mathrm{v}}\xcirc \partial_{\mathrm{h}}+\partial_{\mathrm{h}}\xcirc \partial_{\mathrm{v}}+\partial_{\mathrm{v}}\xcirc \partial_{\mathrm{v}}=0$; while, by Lemma~\ref{lema partial v D} we know that $\partial_{\mathrm{v}}\xcirc D=0$. So it remains to prove that
\begin{equation*}\label{suma de cuatro}
(\partial_{\mathrm{h}} +D)\xcirc D+ D\xcirc \partial_{\mathrm{\mathrm{h}}}+D\xcirc\partial_{\mathrm{v}}=0.
\end{equation*}
But this follows immediately from Lemmas~\ref{lema partial h D}, \ref{lema D partial v} and~\ref{lema D partial h}.
\end{proof}

\begin{notation}
We let $\ho_{\blackdiamond,\Yleft}^*(H,I)$ denote the cohomology groups of  $(\wh{C}_N^*(H,I),\partial+D)$.
\end{notation}

\begin{proposition}\label{significado de cociclos caso general} A pair $(\beta,f)\in \wh{C}_N^2(H,I)=\wh{C}_N^{02}(H,I)\oplus \wh{C}_N^{11}(H,I)$ is a cocycle of $(\wh{C}_N^*(H,I),\partial+D)$ if and only if conditions~\eqref{condicion de cociclo},~\eqref{normalidad y cociclo abeliano},~\eqref{segunda equivalencia de (2.1)1} and~\eqref{condicion equivalente compatibilidad suma cdot con hmassimp3 nuevo} are satisfied.
\end{proposition}

\begin{proof} This follows by a direct computation.
\end{proof}

\begin{corollary}\label{equiv caso general} When $I$ is trivial, there is a bijective correspondence
$$
\Ext_{\blackdiamond,\Yleft}(H,I) \longleftrightarrow \ho_{\blackdiamond,\Yleft}^2(H,I).
$$
\end{corollary}

\begin{proof}
Mimic the proof of Corollary~\ref{equivalencia}.
\end{proof}

\begin{figure}[h]
\begin{center}
\adjustbox{scale=1.4,center}{%
\begin{tikzcd}[row sep=2em, column sep=2 * 1.61803398874988em, scale=2]
\vdots \arrow[rrrrddd, green(ncs), shift left=0.2ex, end anchor={[yshift=0.4ex]}, end anchor={[xshift=-0.1ex]}, start anchor={[xshift=0.5ex]}, start anchor={[yshift=0.3ex]}, outer sep=-1.6pt, "D" pos=0.07] & \vdots \arrow[rrrrddd, green(ncs), shift left=0.2ex, end anchor={[yshift=1.4ex]north west}, end anchor={[xshift=-1.6ex]}, start anchor={[xshift=0.5ex]}, start anchor={[yshift=0.3ex]}, outer sep=-1.6pt, "D" pos=0.07] & \vdots \arrow[rrrdd, dash, green(ncs), shift left=0.3ex, end anchor={[yshift=-1.25ex]north west}, start anchor={[xshift=0.3ex]}, start anchor={[yshift=0.05ex]}, outer sep=-1.6pt, "D" pos=0.105] & \vdots \arrow[rrd, dash, green(ncs), shift left=0.2ex, end anchor={[yshift=-3.2ex]north west}, start anchor={[xshift=0.5ex]}, start anchor={[yshift=0.05ex]}, outer sep=-1.6pt, "D" pos=0.18]  & \vdots \arrow[rd, dash, green(ncs), shift left= 0.4ex, end anchor={[yshift=2.4ex]north west}, outer sep=-1.6pt, "D" pos=0.35] & \\
\wh{C}_N^{03} \arrow[r] \arrow[u] \arrow[rrrdd, green(ncs), shift left=0ex, end anchor={[yshift=0.6ex]}, start anchor={[xshift=-0.3ex]}, start anchor={[yshift=0.3ex]}, outer sep=-1.6pt, "D" pos=0.08] & \wh{C}_N^{13} \arrow[r, crossing over] \arrow[u, crossing over] \arrow[rrrdd, green(ncs), shift left=0ex, end anchor={[yshift=0.6ex]}, start anchor={[xshift=-0.3ex]}, start anchor={[yshift=0.3ex]}, outer sep=-1.6pt, "D" pos=0.08]& \wh{C}_N^{23} \arrow[r, crossing over] \arrow[u, crossing over] \arrow[rrrdd, green(ncs), shift left=0ex, end anchor={[yshift=1.45ex]north west}, start anchor={[xshift=-0.3ex]}, start anchor={[yshift=0.3ex]}, outer sep=-1.6pt, "D" pos=0.08] & \wh{C}_N^{33} \arrow[r, crossing over] \arrow[u, crossing over] \arrow[rrd, dash, green(ncs), shift left=0ex, end anchor={[yshift=-1.25ex]north west}, start anchor={[xshift=-0.3ex]}, start anchor={[yshift=0.05ex]}, outer sep=-1.8pt, "D" pos=0.14] & \wh{C}_N^{43} \arrow[r, crossing over] \arrow[u, crossing over] \arrow[r, dash, green(ncs), shift left=-0.5ex, end anchor={[yshift=-3.4ex]north west}, outer sep=-1.8pt, "D" pos=0.35] & \dots \\
\wh{C}_N^{02} \arrow[r] \arrow[u] \arrow[rrd, green(ncs), shift left=0ex, end anchor={[yshift=0.5ex]}, outer sep=-1.6pt, "D" pos=0.11] & \wh{C}_N^{12} \arrow[r, crossing over] \arrow[u, crossing over]  \arrow[rrd, green(ncs), shift left=0ex, end anchor={[yshift=0.5ex]}, outer sep=-1.6pt, "D" pos=0.12]& \wh{C}_N^{22} \arrow[r, crossing over] \arrow[u, crossing over]  \arrow[rrd, green(ncs), shift left=0ex, end anchor={[yshift=0.5ex]}, outer sep=-1.6pt, "D" pos=0.12]& \wh{C}_N^{32} \arrow[r, crossing over] \arrow[u, crossing over] \arrow[rrd, green(ncs), shift left=0ex, end anchor={[yshift=1.1ex]north west}, outer sep=-1.6pt, "D" pos=0.12] & \wh{C}_N^{42} \arrow[r, crossing over] \arrow[u, crossing over] \arrow[r, dash, green(ncs), shift left=-0.5ex, end anchor={[yshift=-2.5ex]north west}, outer sep=-1.6pt, "D" pos=0.36] & \dots \\
\wh{C}_N^{01} \arrow[r] \arrow[u] \arrow[r, green(ncs), shift left=0.5ex, outer sep=-1.4pt, "D" pos=0.41]  & \wh{C}_N^{11} \arrow[r] \arrow[u, crossing over] \arrow[r, green(ncs), shift left=0.5ex, outer sep=-1.4pt, "D" pos=0.43] & \wh{C}_N^{21} \arrow[r] \arrow[u, crossing over] \arrow[r, green(ncs), shift left=0.5ex, outer sep=-1.4pt, "D" pos=0.42] & \wh{C}_N^{31} \arrow[r] \arrow[u, crossing over] \arrow[r, green(ncs), shift left=0.5ex, outer sep=-1.4pt, "D" pos=0.42] & \wh{C}_N^{41} \arrow[r] \arrow[u, crossing over, outer sep=-1.6pt, "D" pos=0.11] \arrow[r, green(ncs), shift left=0.5ex, outer sep=-1.4pt, "D" pos=0.42] & \dots
\end{tikzcd}
}
\end{center}
\caption{$\bigl(\wh{C}_N^{**}(H,I),\partial_{\mathrm{h}}, \partial_{\mathrm{v}}, D\bigr)$}
\label{ddiagonales}
\end{figure}

\end{document}